\documentclass[a4paper,10pt]{amsart}
\usepackage{amsmath, amssymb, amsthm, amscd}
\usepackage{indentfirst}
\usepackage{paralist}
\usepackage{fullpage}
\usepackage{float,tabu}
\usepackage[all]{xy}
\usepackage[alphabetic, abbrev, lite]{amsrefs}
\usepackage{hyperref} 

\makeatletter
\renewcommand{\BibLabel}{
    \Hy@raisedlink{\hyper@anchorstart{cite.\CurrentBib}\hyper@anchorend} 
    [\thebib]
}
\makeatother

\numberwithin{equation}{section}
\numberwithin{table}{section}

\newcommand{\A}{\ensuremath{\mathbb{A}}}
\newcommand{\boldA}{\ensuremath{\boldsymbol{A}}}
\newcommand{\C}{\ensuremath{\mathbb{C}}}
\renewcommand{\P}{\ensuremath{\mathbb{P}}}
\newcommand{\Z}{\ensuremath{\mathbb{Z}}}
\newcommand{\F}{\ensuremath{\mathbb{F}}}
\newcommand{\Q}{\ensuremath{\mathbb{Q}}}
\newcommand{\R}{\ensuremath{\mathbb{R}}}
\newcommand{\Zhat}{\ensuremath{\hat{\mathbb{Z}}}}
\newcommand{\Fhat}{\ensuremath{\hat{F}}}

\newcommand{\G}{\ensuremath{\mathbb{G}}}
\renewcommand{\S}{\ensuremath{\mathbb{S}}}
\newcommand{\GL}{\ensuremath{\mathrm{GL}}}
\newcommand{\SL}{\ensuremath{\mathrm{SL}}}
\renewcommand{\O}{\ensuremath{\mathcal{O}}}
\newcommand{\SO}{\ensuremath{\mathrm{SO}}}
\newcommand{\GSpin}{\ensuremath{\mathrm{GSpin}}}
\newcommand{\Spin}{\ensuremath{\mathrm{Spin}}}
\newcommand{\Sp}{\ensuremath{\mathrm{Sp}}}
\newcommand{\GSp}{\ensuremath{\mathrm{GSp}}}

\newcommand{\U}{\ensuremath{\mathrm{U}}}
\newcommand{\dd}{\ensuremath{\mathrm{d}}}
\newcommand{\md}{\ensuremath{~\mathrm{d}}}
\renewcommand{\div}{\ensuremath{\mathrm{div}\,}}

\newcommand{\Aut}{\ensuremath{\mathrm{Aut}}}
\newcommand{\End}{\ensuremath{\mathrm{End}}}
\newcommand{\Res}{\ensuremath{\mathrm{Res}}}

\newcommand{\ord}{\ensuremath{\mathrm{ord}}}

\newcommand{\im}{\ensuremath{\operatorname{Im}}}
\newcommand{\re}{\ensuremath{\operatorname{Re}}}
\newcommand{\tr}{\ensuremath{\operatorname{tr}}}
\newcommand{\p}{\ensuremath{\mathfrak{p}}}
\newcommand{\q}{\ensuremath{\mathfrak{q}}}

\newcommand{\f}{\ensuremath{\mathfrak{f}}}
\renewcommand{\u}{\ensuremath{\mathfrak{u}}}
\renewcommand{\v}{\ensuremath{\mathfrak{v}}}
\newcommand{\w}{\ensuremath{\mathfrak{w}}}
\renewcommand{\a}{\ensuremath{\mathfrak{a}}}
\newcommand{\B}{\ensuremath{\mathfrak{B}}}

\newcommand{\rhobar}{\ensuremath{\bar{\rho}}}

\newcommand{\M}{\ensuremath{\mathcal{M}}}

\newcommand{\calF}{\ensuremath{\mathcal{F}}}
\renewcommand{\L}{\ensuremath{\mathcal{L}}}
\newcommand{\D}{\ensuremath{{\mathbb{D}}}}
\renewcommand{\H}{\ensuremath{{\mathbb{H}}}}
\renewcommand{\matrix}[4]{\begin{pmatrix} #1 & #2 \\ #3 & #4 \end{pmatrix}}

\newcommand{\E}{\ensuremath{\mathcal{E}}}
\newcommand{\Schw}{\ensuremath{S(V(\mathbb{A}_f))}}

\newcommand{\sig}{\ensuremath{\mathrm{sig}}}
\newcommand{\Ad}{\ensuremath{\mathrm{Ad}}}
\newcommand{\diag}{\ensuremath{\mathrm{diag}}}

\newcommand{\reg}{\ensuremath{\mathrm{reg}}}
\newcommand{\Pet}{\ensuremath{\mathrm{Pet}}}
\newcommand{\Ind}{\ensuremath{\mathrm{Ind}}}
\newcommand{\CT}{\ensuremath{\mathrm{CT}}}
\newcommand{\CM}{\ensuremath{\mathrm{CM}}}
\newcommand{\Diff}{\ensuremath{\mathrm{Diff}}}
\newcommand{\vecid}{\ensuremath{\mathbf{1}}}
\newcommand{\vectau}{\ensuremath{\vec{\tau}}}

\newcommand{\va}{\ensuremath{\mathfrak{x}}}
\newcommand{\vb}{\ensuremath{\mathfrak{y}}}

\newcommand{\eps}{\ensuremath{\varepsilon}}
\newcommand{\bs}{\backslash}
\renewcommand{\mod}{\ensuremath{\mathrm{mod}}}
\newcommand{\norm}{\operatorname{N}}
\newcommand{\partialbar}{\ensuremath{\bar{\partial}}}

\newcommand{\omegabar}{\ensuremath{\bar{\omega}}}
\newcommand{\tgamma}{\ensuremath{\tilde{\gamma}}}
\newcommand{\tGamma}{\ensuremath{\tilde{\Gamma}}}
\newcommand{\tildeV}{\ensuremath{\widetilde{V}}}
\newcommand{\tildeD}{\ensuremath{\widetilde{D}}}
\newcommand{\tildeG}{\ensuremath{\widetilde{G}}}
\newcommand{\tildeT}{\ensuremath{\widetilde{T}}}
\newcommand{\tildeW}{\ensuremath{\widetilde{W}}}
\newcommand{\tildeE}{\ensuremath{\widetilde{E}}}
\newcommand{\tildeF}{\ensuremath{\widetilde{F}}}
\newcommand{\tDelta}{\ensuremath{\widetilde{\Delta}}}
\newcommand{\tsigma}{\ensuremath{\widetilde{\sigma}}}

\newcommand{\thetaS}{\ensuremath{\theta^{\mathrm{S}}}}
\newcommand{\thetaH}{\ensuremath{\theta^{\mathrm{H}}}}

\newtheorem{prop}{Proposition}[section]
\newtheorem{thm}[prop]{Theorem}

\newtheorem{cor}[prop]{Corollary}

\newtheorem{lem}[prop]{Lemma}

\theoremstyle{definition}
\newtheorem{defn}[prop]{Definition}

\newtheorem{rmk}[prop]{Remark} 

\theoremstyle{remark}

\numberwithin{equation}{section}

\title{\small\textbf{CM Values of Green Functions Associated to Special Cycles on Shimura Varieties with Applications to Siegel 3-Fold $X_2(2)$}}
\author{PENG YU}

\begin{document}

\begin{abstract}
We generalize the definition of CM cycles beyond the ’small’ and ’big’ CM ones studied by various authors, such as in \cite{BY09} and \cite{BKY12} and give a uniform formula for the CM values of Green functions associated to these special cycles in general using the idea of regularized theta lifts. Finally, as an application to Siegel 3-fold $X_2(2)$, we can compute special values of theta functions and Rosenhain $\lambda$-invariants at a CM cycle, which is useful for genus two curve cryptography as in \cite{CDSLY14}.
\end{abstract}

\maketitle

\tableofcontents

\section{Introduction}
In 1985, Gross and Zagier discovered a beautiful factorization formula for singular moduli. This has inspired a lot of interesting work, one of which is to study the 'small' or 'big' CM values of special functions on Shimura varieties. Here 'small' means that the CM cycle is associated to an imaginary quadratic field, while 'big' means that the CM cycles are associated to a maximal torus of the reductive group of the Shimura datum. In his proof of factorization formula on 'small' CM values of Borcherds modular functions on Shimura varieties of orthogonal type, Schofer \cite{Sch09} used the idea of regularized theta lifts. The idea was later adapted by Bruinier and Yang \cite{BY09} to study 'small' CM values of automorphic Green functions and then by Bruinier, Kudla and Yang \cite{BKY12} to study 'big' CM values. 

There are two natural questions. Is there any CM cycle on Shimura variety between 'small' and 'big' CM cycles? Is there a uniform formula for the CM values for all of them? They are motivations for our paper.

We answer the question in our main theorem by using the same idea. There are several applications to our theorem, one of which is to genus two curve. In \cite{CDSLY14}, the authors developed a method for cryptography using classical theta functions and Rosenhain $\lambda$-invariants. By applying our main theorem to Siegel 3-fold case, we can compute special values of theta functions and Rosenhain $\lambda$-invariants at a CM cycle, and give bound to the denominator of the CM value of the Rosenhain $\lambda$-invariants. 

Let $(V,Q_V)$ be a rational quadratic space of signature $(n,2)$, $G=\GSpin(V)$ and $K\subset G(\A_f)$ be a compact open subgroup. Let $\D$ be associated hermitian symmetric domain of oriented negative 2-planes in $V(\R)=V\otimes_{\Q}\R$, and let $X_K$ be the canonical model of Shimura variety over $\Q$ associated to Shimura datum $(G,\D)$ whose $\C$-points are
$$X_K(\C)=G(\Q)\bs (\D\times G(\A_f)/K).$$

Let $d\leq n/2$ and assume there is a totally real number field $F$ of degree $d+1$ and a 2-dimensional $F$-quadratic space $(W,Q_W)$ of signature $\sig(W)=((0,2),(2,0),\dots,(2,0))$
with respect to the $d+1$ $\R$-embeddings $\{\sigma_j\}_{j=0}^d$ such that there exists a positive definite subspace $(V_0,Q_V|_{V_0})$ of $(V,Q_V)$ of dimension $n-2d$ satisfying
\begin{eqnarray*}
V&\cong&V_0\oplus\Res_{F/\Q}W,\\
Q_V(x)&=&Q_V(x_0)+\tr_{F/\Q}Q_W(x_W),\nonumber
\end{eqnarray*}
The negative 2-plane $W_{\sigma_0}$ gives rise to two points $z_0^{\pm}$ in $\D$ with two orientations. $W$ has its associated CM field $E=F(\sqrt{-\det W})$ over $F$.

Let $T=\Res_{F/\Q}\GSpin(W)\to\GSpin(V)=G$, $g\in G(\A_f)$, a special 0-cycle can be defined in $X_K$ according to \cite{Mil90}
$$Z(T,z_0,g)=T(\Q)\bs ({z_0}\times T(\A_f)/K_T^g)\to X_K,\quad [z_0,t]\mapsto [z_0,tg],$$
where $K_T^g$ is the preimage of $gKg^{-1}\subset G(\A_f)$ in $T(\A_f)$. This is the so-called CM cycle.

It is good to note here that when $n$ is even and $d=n/2$, $T$ becomes a maximal torus and this reduces exactly to the case of 'big' CM cycles in \cite{BKY12}. On the other hand, if we allow the case where $d=0$, then $F=\Q$, $W$ will be a 2-dimensional negative definite $\Q$-quadratic space and $E$ is nothing but an imaginary quadratic field, which will reduce to the case of 'small' CM cycles in \cite{BY09}.

Similarly, we can define special divisors on $X_K$. Let $x\in V$ be a vector with $Q_V(x)>0$, $V_x$ be the orthogonal complement of $x$ in V with respect to $Q_V$ and $G_x$ be the stabilizer of $x$ in $G$. Clearly, $G_x\cong\GSpin(V_x)$. The sub-Grassmannian $\D_x=\{z\in\D\mid z\perp x\}$ defines a divisor of $\D$. For $g\in G(\A_f)$, there is an injection
$$G_x(\Q)\bs\D_x\times G_x(\A_f)/(G_x(\A_f)\cap gKg^{-1})\to X_K,\quad [z,g_x]\mapsto [z,g_xg].$$
So its image defines a divisor $Z(x,g)$ on $X_K$. The natural divisor is not stable under pullback of morphism $X_{K_1}\to X_{K_2}$ where $K_1\subset K_2$. To define a better special divisor, let $m\in\Q_{>0}$ and $\varphi\in\Schw^K$, the space of $K$-invariant Schwartz functions on $V(\A_f)$. If we fix an $x_0\in V$ with $Q_V(x_0)=m>0$, we define the following special divisor
$$Z(m,\varphi)=\sum_{g\in G_{x_0}(\A_f)\bs G(\A_f)/K}\varphi(g^{-1}x_0)Z(x_0,g,K).$$

Associated to the quadratic space $(V,Q_V)$ is the reductive pair $(\O(V),\SL_2)$. We know there exists the Weil representation $\omega=\omega_{V,\psi}$ of $\widetilde{\SL_2(\A_f)}$ on $S(V(\A_f))^K$. Using the Weil representation, we can define the Siegel theta function as a linear functional on $S(V(\A_f))$
$$\theta(\tau;z,g)(\varphi)
=v\sum_{x\in V}e(Q_V(x_{z^{\perp}})\tau+Q_V(x_z)\overline{\tau})\otimes\varphi(g^{-1}x),$$
where $x_z$ is the projection of $x$ in the subspace $z$ of $V$, and similarly for $x_{z^\perp}$. $\theta(\tau,z,g)$ is a modular form on $\H$ of weight $n/2-1$ with respect to $\tau$ and an automorphic function on $X_K$ with respect to $[z,g]$. 

In order to define the Green function, we also need the definition of $K$-invariant harmonic weak Maass forms. Let $\rhobar=\omega|_{\widetilde{\SL_2(\Z)}}$ be the Weil representation of $\widetilde{\SL_2(\Z)}$ on $\Schw^K$. A harmonic weak Maass forms is a smooth function $f:\H\to S(V(\A_f))^K$ satisfying certain modular condition with respect to the Weil representation $\rhobar$, the harmonic condition $\Delta_k f=0$, where $\Delta_k$ is the usual weight $k$-hyperbolic Laplacian operator and finally certain growth condition on the cusps. We denote the vector space of all harmonic weak Maass forms of weight $k$ associated with $\rhobar$ by $H_{k,\rhobar}$. The Laurent expansion of $f\in H_{k,\rhobar}$ gives a unique decomposition
$$f(\tau)=f^+(\tau)+f^-(\tau)=\sum_{n\geq n_0}c^+(n)q^n+\sum_{n<0}c^-(n)\Gamma\left(\frac{n}{2},4\pi|n|v\right)q^n,$$
where $\Gamma(a,t)$ denotes the incomplete Gamma function, $v$ is the imaginary part of $\tau\in\H$ and $c^{\pm}(n)\in\Schw^K$. We refer to $f^+$ as the holomorphic part of $f$. In particular, we call $f$ weakly holomorphic if $f^-=0$. Let $M^!_{k,\rhobar}$ be the vector space of all weakly holomorphic modular forms of weight $k$ associated with $\rhobar$.

Now we consider the regularized theta integral as a limit of truncated integrals as follows
$$\Phi(z,g;f)=\int_{\calF}^{\reg}\langle f(\tau),\theta(\tau,z,g)\rangle\md\mu(\tau)
=\CT\left[\lim_{T\to\infty}\int_{\calF_T}\langle f(\tau),\theta(\tau,z,g)\rangle v^{-s}\md\mu(\tau)\right],$$
where $\theta(\tau,z,g)$ is the Siegel theta function defined above, CT stands for the constant term of the Laurent series at $s=0$ and $\calF_T$ is the truncated fundamental domain with imaginary part not more than $T$.

This theta lift was first studied by Borcherds \cite{Bor98} for weekly holomorphic modular forms, and later Bruinier and Funke \cite{BF04} generalized the lift to make it work on harmonic weak Maass forms and most importantly proved that
$\Phi(z,g,f)$ is a Green function for the divisor $Z(f)=\sum_{m>0}Z(m,c^+(-m))$ in the sense of Arakelov geometry in the normalization of \cite{SABK92}. 

Associated to the quadratic subspace $V_0$ is a holomorphic modular form $\theta_0(\tau)$ of weight $\frac{n}{2}-d$ valued in $S(V_0(\A_f))^{\vee}$ defined in a similar way as the Siegel theta function $\theta(\tau,z,g)$.  Associated to $W$ is an incoherent Hilbert Eisenstein series $E_W(\vec{\tau},s, \vecid)$ of weight $(1,\cdots,1)$ on $F$ valued in $S(W(\A_{F,f}))^{\vee}$. This Eisenstein series is automatically zero when $s=0$ and let $\mathcal{E}_W(\tau)$ be the `holomorphic part' of $E_W'(\tau^{\Delta},0,\vecid)$ as defined in \cite{BKY12}. 

For a function $\Theta$ on $X_K$, we define its value on CM cycles to be
$$\Theta(Z(T,h_0,g))=\frac{2}{w_{K,T,g}}\sum_{t\in T(\Q)\bs T(\A_f)/K_T^g}\Theta(h_0,tg).$$
Then we have the following main theorem on CM values of Green function. 

\begin{thm}
\label{thm:intro}
Let $\chi:F^{\times}\backslash\A_F^{\times}\to\C^{\times}$ be the quadratic Hecke character associated to $E/F$ and
$$\Lambda(s,\chi)=\norm_{F/\Q}^{\frac{s}{2}}(\partial_F d_{E/F})\left(\pi^{-\frac{s+1}{2}}\Gamma\left(\frac{s+1}{2}\right)\right)^{d+1}L(s,\chi)$$
be the normalized L-function, where $\partial_F$ is the different of $F$, $d_{E/F}$ is the relative discriminant of $E/F$. Note that $\Lambda(0,\chi)=L(0,\chi)$. 

Then for a $K$-invariant harmonic weak Maass form $f$ of weight $1-\dfrac{n}{2}$ with principle part $f^+$, we have
$$\Phi(Z(W),f)=\frac{\deg(Z(T,z_0^{\pm}))}{\Lambda(0,\chi)}\left(\CT[\langle f^+,\theta_0\otimes\E_W\rangle]-\L^{*,\prime}_W(0,\xi(f))\right),$$
where
$Z(W)$ is the sum of Galois conjugates of $Z(T,z_0^{\pm},1)$ and
\begin{eqnarray*}
\L_W(s,g)&=&\langle g(\tau),\theta_0\otimes E_W(\tau^{\bigtriangleup},s,\vecid)\rangle_{\Pet},\\
\L^*_W(s,g)&=&\Lambda(s+1,\chi)\L_W(s,g)
\end{eqnarray*}
are a Rankin-Selberg convolution $L$-function and its completion for a cusp form $g$ of weight $\dfrac{n}{2}$.
\end{thm}

The special case when  $f$ is weakly holomorphic is of special interest.  According to Borcherds \cite{Bor98}, there is a unique (up to a constant of modulus 1) meromorphic form $\Psi(z,g,f)$ on $X_K$ on $G=\GSpin(V)$ of weight $c^+(0,0)$ satisfying
$$\div\Psi(f)=Z(f),\text{ and }-\log ||\Psi(z,g;f)||_{\Pet}^2=\Phi(z,g;f),$$
where $||\cdot||_{\Pet}^2$ is the normalized Petersson metric. We also know $\xi(f)=0$ and $f^+ =f$ in this case. So we have

\begin{cor}
\label{cor:intro}
Let $f \in M_{1-\frac{n}2, \rhobar}^!$ be a $K$-invariant weakly holomorphic modular form, and let $\Psi(z,g;f)$ be the associated Borcherds lifting of $f$ as above. Then 
\begin{equation}
-\log ||\Psi(Z(W),f)||_{\Pet}^2=\frac{\deg(Z(T,z_0^{\pm}))}{\Lambda(0,\chi)}\CT[\langle f,\theta_0\otimes\E_W\rangle].
\end{equation}
\end{cor}

Next, the first non-trivial example that has not been covered by previous studies of 'big' and 'small' CM cycles is the case where $n=3$, $d=1$. It is well-known that Siegel 3-folds are special cases of orthogonal Shimura varieties of signature $O(3,2)$. So our next attempt is to apply our main theorem to this case. 

Classically, Siegel upper half plane is defined as
$$\H_2=\{\tau\in M_{2\times 2}(\C)\mid \tau\text{ symmetric and }\im(\tau)\text{ positive definite}\}$$
and symplectic group $\Sp_4(\Q)$ acts on $\H_2$ by fractional linear transformation. One can define congruence subgroups
$$\Gamma_2(N)=\ker(\Sp_4(\Z)\to\Sp_4(\Z/N\Z)).$$

For the application to Siegel 3-fold, we actually mean the following quotient space
$$X_2(2)=\Gamma_2(2)\backslash\H_2.$$

In order to apply our main theorem, we also need to realize $X_2(2)$ in our setup of orthogonal Shimura variety. Please check Sections \ref{sec:realization} and \ref{sec:identification} for details. In particular, we identify $W$ with $\tildeE$, where $E$ is a quartic CM field with totally real subfield $F=\Q(\sqrt{D})$ and CM type $\Sigma=\{\sigma_1,\sigma_2\}$ and $\tildeE$ is the reflex field of $(E,\Sigma)$.

Our next focus is to give the CM 0-cycles $Z(W)$ defined above on $X_2(2)$ a moduli description.

First, it is known that $X_2(2)$ parametrizes principally polarized abelian schemes $\boldA=(A,\lambda,\psi:A[2]\stackrel{\sim}{\to}(\Z/2\Z)^4)$ with $A$ an abelian scheme of relative dimension 2 over $\C$ with 2-torsion $A[2]$, a principally polarization $\lambda:A\to A^{\vee}$ satisfying some inherent conditions, and $\psi$ preserves the symplectic forms between the Weil pairing on $A[2]\times A^{\vee}[2]$ and the standard symplectic pairing on $(\Z/2Z)^4$. 

Next, if we define $\CM_2^{\Sigma}(E)$ to be the set of isomorphic classes of principally polarized CM abelian schemes $\boldA=(A,\kappa,\lambda,\psi:A[2]\stackrel{\sim}{\to}(\Z/2\Z)^4)$ of CM type $(\O_E,\Sigma)$ with an $\O_E$-action $\kappa:\O_E\hookrightarrow\End(A)$, there will be a map 
$$j:\CM_2(E)=\coprod_{\Sigma}\CM_2^{\Sigma}(E)\to X_2(2).$$
which defines a CM point on $X_2(2)$. 

Via this moduli description, we are able to give $z_0^{\pm}$ a geometric interpretation on Siegel upper half plane. Please check Section \ref{sec:geom} for details. Let $T=\{z\in E^{\times}\mid z\bar{z}\in\Q^{\times}\}$. We could construct a CM cycle
$$Z(W)=Z(\boldA)=T(\Q)\bs\{z_0^{\pm}\}\times T(\A_f)/K_T,$$
which has a $C(T)$-action with $C(T)=T(\Q)\bs T(\A_f)/K_T$. 

Now let $z\in\H_2$ and a quadruple $(\va,\vb)=(x_1,x_2,y_1,y_2)\in\{0,1\}^4$, the Siegel theta constant of characteristic $(\va,\vb)$ is defined as
$$\thetaS_{\va,\vb}(z)=\sum_{m\in\Z^2}\exp\left(\pi i\left(m+\frac{\va}{2}\right)z\left(m+\frac{\va}{2}\right)^t+2\pi i\left(m+\frac{\va}{2}\right)\left(\frac{\vb}{2}\right)^t\right).$$

The function $\thetaS_{\va,\vb}(z)$ is a Siegel modular form of weight $\dfrac{1}{2}$ for the modular group $\Gamma_2(2)$. A quadruple $(\va,\vb)$ as above is called even if $\va^t\vb=0$, i.e. $x_1y_1+x_2y_2\equiv 0~(\mod~2)$. There are 10 even quadruples and it is well-known that $\thetaS_{\va,\vb}\not=0$ if and only if $(\va,\vb)$ is even.

For each even pair $(\va,\vb)$, Lippolt constructed in \cite{Lip08} an weakly holomorphic modular form $f_{\va,\vb}$ of weight $-1/2$ valued in $S_L$ such that 
$$\theta_{\va, \vb}^S(z)=\Psi(z, f_{\va,\vb})\text{ or }-\log\|\thetaS_{\va,\vb}(z)\|_{\Pet}^2=\Phi(z,f_{\va,\vb})$$
is the Borcherds lifting of $f_{\va,\vb}$ in the fashion of Theorem \ref{thm:Borcherds}. Then a direct consequence of Corollary \ref{cor:intro} is
\begin{thm}
Notation as above. Let $L$ be an even integral lattice in $V$, $L_0=L\cap V_0$ and $M=L\cap\Res_{F/\Q}W$. Let $L',~L_0',~M'$ be their corresponding duals. Then there is a surjection $\varpi:L'/(L_0\oplus M)\to L'/L$ and we have
$$-\log\|\theta_{\va,\vb}^S(Z(\boldA))\|_{\Pet}^2=C_E\cdot\CT\left[\sum_{\mu\in L'/L}\sum_{(\mu_0,\mu_1)\in\varpi^{-1}(\mu)}f_{\va,\vb,\mu}\theta_0(\varphi_{\mu_0,L_0})\E_W(\varphi_{\mu_1,M})\right],$$
where 	$$C_E=\frac{\deg(Z(\boldA))}{\Lambda(0,\chi)}=\frac{4}{\omega_E}\frac{|C(T)|}{\Lambda(0,\chi)},$$
with $\omega_E$ being the number of roots of unity in $E$ and $\CT[\sum_{n\in\Z} a_nq^n]=a_0$.
\end{thm}

\begin{rmk}
Here I use the lattice version of main Theorem \ref{thm:lattice} to derive the formula for purpose of computation. Please check Section \ref{sec:lattice} for details of $\varpi$, $f_{\va,\vb,\mu}$, $\theta_0(\varphi_{\mu_0,L_0})$ and $\E_W(\varphi_{\mu_1,M})$.
\end{rmk}

Furthermore, we can also derive a formula for special values $\lambda_k(Z(\boldA))$, $k=1,2,3$ for Rosenhain lambda invariants at these CM cycles in Theorem \ref{thm:MainF2}. In particular, the formula can be used to give bounds on the denominators in the coefficients of Rosenhain lambda invariants' minimal polynomials, which is essential in the study of genus two curve cryptography as in \cite{CDSLY14}. We also would like to remark that other algebraic properties of Rosenhain lambda invariants are also quite interesting, as the author has explored thoroughly that of modular lambda invariants with Yang and Yin in \cite{YYY18}. The study of Rosenhain lambda invariants is more complex and challenging and we would like to present it in a sequel.

Finally, we also show that similar results are also true for unitary Shimura varieties of type $(n,1)$ in Section \ref{cha:unitary}.

\section{Orthogonal Shimura Varieties}
\label{cha:SV}
For a number field $F$, we denote the ad\`eles of $F$ by $\A_F$ and the finite ad\`eles by $\Fhat=F\otimes_{\Z}\Zhat$. In particular, if $F=\Q$, we simply write $\A$ for $\A_{\Q}$ and $\A_f$ for the finite ad\`eles of $\Q$. Let $(V,Q_V)$ be a rational quadratic space of signature $(n,2)$ for some positive integer $n$. Let $G=\GSpin(V)$ be the general Spin group of $V$ over $\Q$ which satisfies the exact sequence
$$1\to\G_m\to G=\GSpin(V)\to\SO(V)\to 1$$
and let $K\subset G(\A_f)$ be a compact open subgroup. Let $\D$ be the associated hermitian symmetric domain of oriented negative 2-planes in $V(\R)=V\otimes_{\Q}\R$, and let $X_K$ be the canonical model of Shimura variety over $\Q$ associated to Shimura datum $(G,\D)$ whose $\C$-points are
$$X_K(\C)=G(\Q)\bs (\D\times G(\A_f)/K).$$

Let $d\leq n/2$ be a non-negative integer and assume there is a totally real number field $F$ of degree $d+1$ and a 2-dimensional $F$-quadratic space $(W,Q_W)$ of signature
$$\sig(W)=((0,2),(2,0),\dots,(2,0))$$
with respect to the $d+1$ $\R$-embeddings $\{\sigma_j\}_{j=0}^d$ such that there exists a positive definite subspace $(V_0,Q_V|_{V_0})$ of $(V,Q_V)$ of dimension $n-2d$ satisfying
\begin{eqnarray}
\label{eqn:decomposition}
V&\cong&V_0\oplus\Res_{F/\Q}W,\\
Q_V(x)&=&Q_V(x_0)+\tr_{F/\Q}Q_W(x_W),\nonumber
\end{eqnarray}
if $x\in V$ maps to $x_0+x_W$ under (\ref{eqn:decomposition}). For abuse of language, we'll simply write $V=V_0\oplus\Res_{F/\Q}W$. For future reference, let us denote $\Res_{F/\Q}W$ by $W_0$. Then there is an orthogonal direct sum decomposition
$$V(\R)=V_0(\R)\oplus\left(\oplus_j W_{\sigma_j}\right),\quad W_{\sigma_j}=W\otimes_{F,\sigma_j}\R$$
with respect to the quadratic form $Q_V$.
The negative 2-plane $W_{\sigma_0}$ gives rise to two points $z_0^{\pm}$ in $\D$ with two orientations. 

\subsection{CM Cycles}
Let $T=\Res_{F/\Q}\GSpin(W)$. There is a homomorphism
\begin{equation}
\label{eqn:algebraic homomorphism}
T=\Res_{F/\Q}\GSpin(W)\to\GSpin(V)=G
\end{equation}
as algebraic groups over $\Q$, whose real points, gives rise to the homomorphism
\begin{equation}
\label{eqn:homomorphism}
T(\R)=\prod_{j}\GSpin(W_{\sigma_j})\to\GSpin(V\otimes_{\Q}\R)=G(\R),
\end{equation}
associated to the decomposition (\ref{eqn:decomposition}). We define $\tildeT$ to be the image of (\ref{eqn:algebraic homomorphism}).

A more explicit description of $T$ can be given as follows using Clifford algebra.

Recall that the Clifford algebra $C(V)$ is defined as the quotient algebra $\otimes V/I(V)$, where $\otimes V$ is the tensor algebra over $V$, $I(V)$ is the ideal generated by all elements of the form
$$v\otimes v-Q(v)1\text{ for all }v\in V.$$
It has a main involution $i(v_1\otimes\cdots\otimes v_m)=v_m\otimes\cdots\otimes v_1$, a degree map $\deg(v_1\otimes\cdots\otimes v_m)=m$, a natural grading $C(V)=C^0(V)\oplus C^1(V)$ and a canonical embedding $V\to C^1(V)$, where $C^k(V)=\{v\in C(V)\mid \deg(v)\equiv k~\mod~2\}$. $C^0(V)$ and $C^1(V)$ are usually called the even and odd Clifford algebras respectively. Therefore, we could define
\begin{eqnarray*}
\GSpin(V)&=&\{g\in C^0(V)^{\times}\mid gi(g)=\nu(g)\text{ and }gVg^{-1}=V\},\\
\Spin(V)&=&\{g\in\GSpin(V)\mid\nu(g)=1\},
\end{eqnarray*}
where $\nu(g)$ is the Spin character.

Actually, $T$ is a torus associated to the CM number field $E=F(\sqrt{-\det W})$. Indeed, we can prove that $T(\Q)=E^{\times}$.

First, if we fix an orthogonal basis $e,f$ for $W$, then the Clifford algebra $C_F(W)$ is 4-dimensional $F$-vector space generated by $1,e,f,ef$, satisfying $e^2=Q_W(e)$, $f^2=Q_W(f)$ and $ef=-fe$. On one hand, the even part $C_F^0(W)$ of $C_F(W)$ is $F$-vector space generated by $1,~ef$ with $(ef)^2=efef=-e^2f^2=-Q_W(e)Q_W(f)$, which is totally negative under all embeddings $F\hookrightarrow\R$. So $C_F^0(W)$ is nothing but pure imaginary quadratic extension $E=F(\sqrt{-Q_W(e)Q_W(f)})=F(\sqrt{-\det W})$. Finally, $T(\Q)=\Res_{F/\Q}\GSpin(W)(\Q)=E^{\times}$ as they are invertible elements in $C^0(W)$ in general. On the other hand, the odd part of the Clifford algebra $C_F^1(W)=W$ is generated by $e,~f$. Since $Q_W(e),~Q_W(f)\in F$, we could view $W=Ee=E$ with $Q_W(x)=Q_W(e)x\bar{x}$.

We will have the following exact sequences
$$\xymatrix{
1\ar[r]	&\Res_{F/\Q}\G_m\ar[r]\ar[d]_{\norm_{F/\Q}}	&T=\Res_{F/\Q}\GSpin(W)\ar[r]\ar@{->>}[d]	&\Res_{F/\Q}\SO(W)\ar[r]\ar[d]_=	&1\\
1\ar[r]	&\G_m\ar[r]\ar[d]_=	&\tildeT\ar[r]\ar[d]	&\Res_{F/\Q}\SO(W)\ar[r]\ar@{^{(}->}[d]	&1\\
1\ar[r]	&\G_m\ar[r]	&G=\GSpin(V)\ar[r]	&\SO(V)\ar[r]	&1\\
}$$
as algebraic groups over $\Q$ with their $\Q$-points on the top row satisfying another exact sequence
$$\xymatrix{
1\ar[r]	&F^{\times}\ar[r]	&E^{\times}\ar[r]	&E^1\ar[r]	&1
}$$

If we fix an identification $\S=\Res_{\C/\R}\G_m\stackrel{\sim}{\to}\GSpin(W_{\sigma_0})$, we obtain a homomorphism $h_0:\S\to G_{\R}$ as algebraic groups over $\R$ corresponding to the inclusion in the first factor in (\ref{eqn:homomorphism}). Let $\{e_0,f_0\}$ be a standard oriented basis of $W_{\sigma_0}\subset V\otimes\R$. Then it is easy to check that the following map
$$gh_0g^{-1}\mapsto\R ge_0+\R gf_0$$
gives a bijection of $G(\R)$-conjugacy class of $h_0$ and $\D$, the set of oriented negative 2-planes in $V\otimes_{\Q}\R$. We will not distinguish between the two interpretations of $\D$. In particular, the two points $z_0^{\pm}$ in $\D$ described earlier correspond to the two different orientations determined by $h_0$, or $\{e_0,f_0\}$, or the two extensions of $\sigma_0:F\hookrightarrow\C$ to $E\hookrightarrow\C$.

By construction, $h_0$ will factor through $T_{\R}$. So we have, for any $g\in G(\A_f)$, a special 0-cycle in $X_K$ according to \cite{Mil90}
\begin{equation}
\label{eqn:CM}
Z(T,h_0,g)_K=T(\Q)\bs ({h_0}\times T(\A_f)/K_T^g)\to X_K,\quad [h_0,t]\mapsto [h_0,tg]
\end{equation}
where $K_T^g$ is the preimage of $gKg^{-1}\subset G(\A_f)$ in $T(\A_f)$. We will usually drop the subscript $K$ and identify $Z(T,h_0,g)$ with its image in $X_K$, but every point in $Z(T,h_0,g)$ is counted with multiplicity $\frac{2}{w_{K,T,g}}$ and $w_{K,T,g}$ is the number of roots of unity in $T(\Q)\cap K_T^g$. In particular, for a function $\Theta$ on $X_K$, we have
\begin{equation}
\label{eqn:multiplicity}
\Theta(Z(T,h_0,g))=\frac{2}{w_{K,T,g}}\sum_{t\in T(\Q)\bs T(\A_f)/K_T^g}\Theta(h_0,tg).
\end{equation}
When $g=1$, we will further abbreviate notation and write $Z(T,h_0)$ for $Z(T,h_0,1)$. 

\begin{rmk}
It is not hard to show that $Z(T,h_0,g)=Z(\tildeT,h_0,g)$, since $\tildeT$ is the image of $T\to G$. For general theory, we are still going to use $T$ for the torus. However, for computation purpose in the application, $\tildeT$ is actually the more useful one.
\end{rmk}

As mentioned in the introduction section, when $n$ is even and $d=n/2$, $T$ becomes a maximal torus and this reduces to the case of 'big' CM cycles in \cite{BKY12}. On the other hand, when $d=0$, then $F=\Q$, $E$ will automatically become an imaginary quadratic field, which will reduce to the case of 'small' CM cycles in \cite{BY09}.

Next, we would like to define a formal sum $Z(W)$ of $Z(T,h_0,g)$ with its all Galois conjugates following \cite{BKY12}.

First, the 0-cycle $Z(T,h_0)$ is defined over $\sigma_0(E)$, the reflex field of $(T,h_0)$. Let us now describe its Galois conjugates.

For $j\in\{0,\dots,d\}$, let $W(j)$ be the unique (up to isomorphism) quadratic space over $F$ such that $W(j)\otimes_F F_v$ and $W\otimes_F F_v$ are isometric for all finite place $v$ of $F$, and that
\begin{equation}
\label{eqn:signature}
\sig(W(j))=((2,0),\dots,(2,0),\underbrace{\hbox{(0,2)}}_{\hbox{j}},(2,0),\dots,(2,0)).
\end{equation}
Note that, although the quadratic spaces $W$ and $W(j)$ over $F$ are not isomorphic for $j\not=0$, there is an isomorphism $C_F^0(W(j))\cong C_F^0(W)=E$ of their even Clifford algebras. Let $V(j)=V_0\otimes\Res_{F/\Q}W(j)$ with quadratic form $Q_{V(j)}(x)=Q_V(x_0)+\tr_{F/\Q}Q_{W_j}(x_W)$ for $x=x_0+x_W$ under the orthogonal decomposition (\ref{eqn:decomposition}). The signature of $V(j)$ is still $(n,2)$ and the quadratic spaces $V$ and $V(j)$ are isomorphic over $\Q$. If we fix an isomorphism
\begin{equation}
\label{eqn:isomorphism}
V(j)\stackrel{\sim}{\to}V,
\end{equation}
we could identify $V(j)$ with $V$. Let $T(j)=\Res_{F/\Q}\GSpin(W(j))$ and $h_0(j):\S\to G_{\R}$ be the homomorphism defined in a similar way how we define $h_0$. As noted above, there is an isomorphism $C_F^0(W(j))\cong C_F^0(W)=E$ of their even Clifford algebras. Therefore, as invertible elements in even Clifford algebra, $T(j)=\Res_{F/\Q}\GSpin(W(j))$ and $T=\Res_{F/\Q}\GSpin(W)$ are isomorphic. For $g\in G(\A_f)$, the analogue of the construction above yields a special 0-cycle $Z(T(j),h_0(j),g)$ on $X_K$ defined over $\sigma_j(E)$.

\begin{rmk}
If $F/\Q$ is Galois, we have the following explicit construction of $W(j)$.

Let $\alpha=Q_W(e_0),~\beta=Q_W(f_0)\in F^{\times}$ satisfying
$$\sigma_i(\alpha)>0,\quad i>0,\quad\sigma_0(\alpha)<0.$$
Then essentially, $W(j)=W$ as an $F$-vector space but with a different quadratic form
$$Q_{W(j)}(e_0)=\alpha_j,\quad Q_{W(j)}(f_0)=\beta_j$$
where $\alpha_j,~\beta_j$ are Galois conjugates of $\alpha,~\beta$, respectively, satisfying
\begin{eqnarray*}
\sigma_i(\alpha_j)>0,&i\not=j,&\sigma_j(\alpha_j)<0,\\
\sigma_i(\beta_j)>0,&i\not=j,&\sigma_j(\beta_j)<0.
\end{eqnarray*}
Hence we have $\tr_{F/\Q}Q_{W(j)}=\tr_{F/\Q}Q_W$. Therefore, the following decomposition holds for any $j$
\begin{eqnarray*}
V&=&V_0\oplus\Res_{F/\Q}W(j),\\
Q_V(x)&=&Q_V(x_0)+\tr_{F/\Q}Q_{W(j)}(x_W)
\end{eqnarray*}
\end{rmk}

We fix an $\Fhat$-linear isometry
$$\mu_j: W(j)(\Fhat)\stackrel{\sim}{\to}W(\Fhat).$$
Noting that there are canonical identifications $W(j)(\Fhat)=W_0(j)(\A_f)$ and $W(\Fhat)=W_0(\A_f)$, and using the fixed identification of $V$ and $V(j)$, there is a unique element $g_{j,0}\in\O(V)(\A_f)$ such that the isometry $\mu_j$ can be expressed by $g_{j,0}^{-1}$ on $V(\A_f)$.

Modifying the isometry $\mu_j$ by an element of $\O(W)(\Fhat)$, if necessary, we can assume that $g_{j,0}\in\SO(V)(\A_f)$. For any element $g_j\in G(\A_f)$ with image $g_{j,0}\in\SO(V)(\A_f)$, the finite ad\`eles of the tori $T(j)$ and $T$ are related, as subgroups of $G(\A_f)$, by
$$T(j)(\A_f)=g_jT(\A_f)g_j^{-1},$$
and hence
$$K_{T(j)}^{g_j}=g_jK_Tg_j^{-1}.$$
These relations depend only on the image $g_{j,0}$ of $g_j$.

The reciprocity laws for the action of $\Aut(\C)$ on special points of Shimura varieties yields the following lemma.

\begin{lem}[\cite{BKY12}*{Lemma 2.2}]
Let the notation be as above and let $\tau\in\Aut(\C/\Q)$.

\begin{enumerate}
\item[(1)] If $\tau=\sigma_j\circ\sigma_0^{-1}$ on $\sigma_0(E)$, then 
$$\tau(Z(T,h_0))=Z(T(j),h_0(j),g_j).$$
\item[(2)] If $\tau$ is complex conjugation, then
$$\tau(Z(T,h_0))=Z(T,h_0^-),$$
where $h_0^-$ is the map from $\S\to G_{\R}$ induced by $\S\to\GSpin(W_{\sigma_0})$, $z\mapsto\overline{z}$.
\end{enumerate}
\end{lem}

We will write 
$$Z(T(j),h_0^{\pm}(j),g_j)=Z(T(j),h_0^+(j),g_j)+Z(T(j),h_0^-(j),g_j).$$

We will also write $z_0^{\pm}(j)\in\D$ for the oriented negative two planes in $V(\R)$ associated to $h_0^{\pm}(j)$. Let 
\begin{equation}
\label{eqn:cycle}
Z(W)=\sum_{j=0}^d Z(T(j),h_0^{\pm}(j),g_j)\in Z^n(X_K).
\end{equation}
Then $Z(W)$ is a CM 0-cycle defined over $\Q$.

\section{Special Divisors}
Let $x\in V$ be a vector with $Q_V(x)>0$, $V_x$ be the orthogonal complement of $x$ in V with respect to $Q_V$ and $G_x$ be the stabilizer of $x$ in $G$. Clearly, $G_x\cong\GSpin(V_x)$. The sub-Grassmannian
$$\D_x=\{z\in\D\mid z\perp x\}$$
defines a divisor of $\D$. For $g\in G(\A_f)$, there is a natural map
\begin{equation}
G_x(\Q)\bs\D_x\times G_x(\A_f)/(G_x(\A_f)\cap gKg^{-1})\to X_K,\quad [z,g_x]\mapsto [z,g_xg].
\end{equation}
This map is actually an injection, so its image defines a divisor $Z(x,g,K)$ on $X_K$. The natural divisor is not stable under pullback of morphism $X_{K_1}\to X_{K_2}$ where $K_1\subset K_2$, so Kudla in \cite{Kud97} defines a special divisor as a weighted sum of natural divisors which has nice properties. To define the special divisor, let $m\in\Q_{>0}$ and $\varphi\in\Schw^K$, the space of $K$-invariant Schwartz functions on $V(\A_f)$. If we fix an $x_0\in V$ with $Q_V(x_0)=m>0$, we define the following special divisor
\begin{equation}
Z(m,\varphi)=\sum_{g\in G_{x_0}(\A_f)\bs G(\A_f)/K}\varphi(g^{-1}x_0)Z(x_0,g).
\end{equation}
It is a divisor on $X_K$ with complex coefficients. Note that, since $\varphi$ has compact support in $V(\A_f)$ and $K$ is open, the sum is actually finite. If there is no $x_0\in V$ such that $Q_V(x_0)=m$, we set $Z(m,\varphi)=0$.

\begin{prop}[\cite{BKY12}*{Proposition 3.1}]
Let notation be as above. Then $Z(m,\varphi)$ and $Z(T(j),h_0^{\pm}(j),g_j)$ do not intersect in $X_K$.
\end{prop}

\section{Weil Representation}
\label{cha:Weil}
Let $(V,(,)_V)$ still be the quadratic space defined above and $(\Q^{\oplus 2},\langle,\rangle)$ the standard 2-dimensional non-degenerate symplectic space over $\Q$. Then $(\Omega=V\otimes_{\Q}\Q^{\oplus 2},\langle,\rangle_{\Omega}=(,)_V\otimes\langle,\rangle$) is also a symplectic space over $\Q$. Let $\widetilde{\SL_{2,\A}}$ (resp. $\widetilde{\Sp_{\Omega,\A}}$) be the metaplectic double cover of $\SL_{2,\A}$ (resp. $\Sp(\Omega)_{\A}$).Then we have the following group homomorphisms on $\A$
$$\xymatrix{
\O(V)_{\A}\times\widetilde{\SL_{2,\A}}\ar[r]\ar[d]&\widetilde{\Sp_{\Omega,\A}}\ar[d]\\
\O(V)_{\A}\times\SL_{2,\A}\ar[r]&\Sp_{\Omega,\A}\\	
}.
$$

If we identify $\Omega=V\otimes_{\Q}\Q^{\oplus 2}=V^{\oplus 2}$. Then by Stone–-von Neumann theorem and construction of Weil representation, for any non-trivial additive character $\psi:\Q\backslash\A\to\C^{\times}$, there exists a unique (up to isomorphism) Weil representation $\omega=\omega_{\psi}$ of $\widetilde{\SL_2(\A)}$ on $S(V(\A))$ with central character $\psi$.

By Weil, there always exists a canonical splitting
\begin{equation}
\label{eqn:splitting}
\eta:\SL_2(\Q)\to\widetilde{\SL_{2,\A}}.
\end{equation}

Let us fix the non-trivial additive character $\psi$ to be the canonical unramified additive character of $\Q\backslash\A$ with $\psi_{\infty}(x)=e^{2\pi ix}$ for $x\in\R$.

Note that we can identify $\widetilde{\SL_2(\R)}$ with
$$\left\{(\gamma,\epsilon)\left|\gamma=\matrix{a}{b}{c}{d}\in\SL_2(\R),\epsilon(\tau)^2=c\tau+d\right.\right\},$$
and the group multiplication is given by
$$(\gamma_1,\epsilon_1(\tau))(\gamma_2,\epsilon_2(\tau))=(\gamma_1\gamma_2,~\epsilon_1(\gamma_2\tau)\epsilon_2(\tau)).$$

Let $K'$ be the full inverse image of $K=\SL_2(\Zhat)$ in $\widetilde{\SL_2(\A_f)}$. Then for any $\tgamma=(\gamma,\epsilon)\in\widetilde{\SL_2(\Z)}\subset\widetilde{\SL_2(\R)}$, $\eta(\gamma)\in\widetilde{\SL_2(\A)}$ under the splitting (\ref{eqn:splitting}) can be written as $\hat{\gamma}\tgamma$, where $\hat{\gamma}=(\gamma,\hat{\epsilon})\in K'$ is the finite ad\`eles part and $\tgamma$ is the infinity ad\`eles part. In other words, there exists a unique $\hat{\gamma}\in K'$, such that $\eta(\gamma)=\tgamma\hat{\gamma}$. Then we can define
\begin{equation}
\label{eqn:Weil}
\rho(\tgamma)\varphi=\omegabar_f(\hat{\gamma})\varphi,
\end{equation}
where $\omega_f$ is the restriction of $\omega$ to $\widetilde{\SL_2(\A_f)}$. Then $\rho$ becomes a representation of $\widetilde{\SL_2(\Z)}$ on $S(V(\A_f))$. And note that the conjugation $\rhobar$ of $\rho$ is thus the restriction of $\omega$ to the subgroup $\widetilde{\SL_2(\Z)}$.

\section{Siegel Theta Functions}
\label{cha:theta}
Associated to the quadratic space $(V,Q_V)$ is the reductive pair $(\O(V),\SL_2)$ and the Weil representation $\omega=\omega_{V,\psi}$ of $\widetilde{\SL_2(\A_f)}$ on $S(V(\A_f))$. Recall in Section \ref{cha:Weil}, for any $\tgamma=(\gamma,\epsilon)\in\widetilde{\SL_2(\Z)}\subset\widetilde{\SL_2(\R)}$, $\eta(\gamma)\in\widetilde{\SL_2(\A)}$ under the splitting (\ref{eqn:splitting}) can be written as $\hat{\gamma}\tgamma$, where $\hat{\gamma}=(\gamma,\hat{\epsilon})$ is the finite ad\`eles part and $\tgamma$ is the infinity ad\`eles part. The groups $\widetilde{\SL_2(\A_f)}$ and $G(\A_f)$ act on the space $S(V(\A_f))$ of Schwartz-Bruhat functions of $V(\A_f)$ via the Weil representation $\omega=\omega_{\psi}$.

For $z\in\D$, one has decomposition
$$V(\R)=z\oplus z^{\perp},\quad x=x_z+x_{z^{\perp}}.$$
Let $(x,x)_z=-(x_z,x_z)_V+(x_{z^{\perp}},x_{z^{\perp}})_V$ and define the associated Gaussian by
$$\varphi_{\infty}(x,z)=e^{-\pi(x,x)_z},$$
which belongs to $S(V(\R))$. It has invariance property $\varphi_{\infty}(gx,gz)=\varphi_{\infty}(x,z)$ for any $g\in G(\R)$. Moreover, it has weight $n/2-1$ under the action of the maximal compact subgroup $K'_{\infty}$, where $K'_{\infty}$ is the full inverse image in $\widetilde{\SL_2(\R)}$ of $K_{\infty}=\SO_2(\R)\subset\SL_2(\R)$.

Then, following \cite{BY09}, for any $\tau=u+iv\in\H$ and $[z,g]\in X_K$, the theta function $\theta_V(\tau, z, g)=\theta(\tau, z, g)$, as a linear functional on $\Schw$, is given by
\begin{eqnarray}
\label{eqn:theta}
\theta(\tau,z,g)(\varphi)&=&\sum_{x\in V}w(g_{\tau}')\varphi_{\infty}(x,z)\varphi(g^{-1}x),\\
&=&v\sum_{x\in V}e(Q_V(x_{z^{\perp}})\tau+Q_V(x_z)\overline{\tau})\otimes\varphi(g^{-1}x),\nonumber
\end{eqnarray}
where $g_{\tau}'=(g_{\tau},v^{-\frac{1}{4}})\in\widetilde{\SL_2(\R)}$ with
$$g_{\tau}=\left(\begin{array}{cc}v^{\frac{1}{2}}&uv^{-\frac{1}{2}}\\0&v^{-\frac{1}{2}}\end{array}\right)\in\SL_2(\R).$$
Here $g$ acts on $V$ via its image in $\SO(V)$. The theta kernel function $\theta(\tau,z,g)(\varphi)$ is a modular form of weight $n/2-1$ with respect to $\tau$ and an automorphic function on $X_K$ with respect to $[z,g]$.

\begin{lem}
As a function on $S(V(\A_f))=S(V_0(\A_f))\otimes S(W(\Fhat))$, we have
$$\theta_V=\theta_0\otimes\theta_W,$$
where $\theta_V,~\theta_0$ and $\theta_W$ are theta kernel functions associated to $V,~V_0$ and $W$ respectively.
\end{lem}

\begin{proof}
Let $\varphi_{\infty,V}$ and $\varphi_{\infty,W}$ be the same function $\varphi_{\infty}$ defined above associated to $V$ and $W$ respectively, and let $\varphi_{\infty,V_0}$ be the Gaussian $e^{-\pi(x_0,x_0)_V}$ for $x_0\in V_0$. Then one has
$$\varphi_{\infty,V}(x,z)=\varphi_{\infty,V_0}(x_0)\varphi_{\infty,W}(x_W),\text{ if }x=x_0+x_W\text{ under }(\ref{eqn:decomposition}).$$
So for $\varphi=\varphi_0\otimes\varphi_W\in S(V_0(\A_f))\otimes S(W(\Fhat))=S(V(\A_f))$, one has for $\tau\in\H$, $g_0\in\GSpin(V_0)(\A_f)$ and $[z,g]\in Z(W)$,
$$\theta_V(\tau,z,g_0g)(\varphi)=\theta_0(\tau,g_0)(\varphi_0)\theta_W(\tau,z,g)(\varphi_W).$$
\end{proof}

\section{Automorphic Green Functions}
\label{cha:Green}
\subsection{Harmonic Weak Maass Forms}
\begin{defn}[Harmonic Weak Maass Forms]
A smooth function $f:\H\to\Schw$ is called a harmonic weak Maass form of weight $k$ with respect to $\widetilde{\SL_2(\Z)}$ and  the Weil representation $\rhobar=\omega|_{\widetilde{\SL_2(\Z)}}$of $\widetilde{\SL_2(\Z)}$ on $\Schw$ if the following is satisfied:
\begin{enumerate}
\item $f|_{k,\rhobar}\tgamma=f$ for all $\tgamma=(\gamma,\epsilon)\in\widetilde{\SL_2(\Z)}$, where 
$$f|_{k,\rhobar}\tgamma(\tau)=\epsilon(\tau)^{-2k}(\rhobar(\tgamma))^{-1}f(\gamma\tau),$$
i.e. we have
$$f(\gamma\tau)=\epsilon(\tau)^{2k}\rhobar(\tgamma)f(\tau);$$
\item There is an $\Schw$-valued Fourier polynomial
$$P_f(\tau)=\sum_{n\leq 0}c^+(n)q^n,$$
where $c^+(n)\in\Schw$, such that $f(\tau)-P_f(\tau)=O(e^{-\eps v})$ as $v\to\infty$ for some $\eps>0$;
\item $\Delta_k f=0$, where
$$\Delta_k:=-v^2\left(\frac{\partial^2}{\partial u^2}+\frac{\partial^2}{\partial v^2}\right)+ikv\left(\frac{\partial}{\partial u}+i\frac{\partial}{\partial v}\right)$$
is the usual weight $k$ hyperbolic Laplacian operator.
\end{enumerate}
\end{defn}
The Fourier polynomial $P_f$ is called the principal part of $f$. We denote the vector space of all harmonic weak Maass forms of weight $k$ associated with $\rhobar$ by $H_{k,\rhobar}$. $f\in H_{k,\rhobar}$ has the following Fourier expansion
\begin{equation}
\label{eqn:Fourier}
f(\tau)=f^+(\tau)+f^-(\tau)=\sum_{n\geq n_0}c^+(n)q^n+\sum_{n<0}c^-(n)\Gamma\left(\frac{n}{2},4\pi|n|v\right)q^n,
\end{equation}
where $\Gamma(a,t)=\int_t^{\infty} x^{a-1}e^{-x}\md x$ is the incomplete Gamma function, $v$ is the imaginary part of $\tau\in\H$ and $c^{\pm}(n)\in\Schw$. We refer to $f^+$ as the holomorphic part and $f^-$ as the non-holomorphic part of $f$. In particular, we call $f$ weakly holomorphic if $f^-=0$. Let $M^!_{k,\rhobar}$ be the vector space of all weakly holomorphic modular forms of weight $k$ associated with $\rhobar$.

Recall that there is an anti-linear differential operator $\xi=\xi_k:H_{k,\rhobar}\to S_{2-k,\rho}$, defined by
\begin{equation}
f(\tau)\mapsto\xi(f)(\tau):=2iv^k\overline{\frac{\partial}{\partial\bar{\tau}}f(\tau)},
\end{equation}
here $S_{2-k,\rho}$ stands for the vector space of all cusp forms of weight $2-k$ associated with $\rho$.

By \cite{BF04}, one has an exact sequence
\begin{equation}
0\to M_{k,\rhobar}^!\to H_{k,\rhobar}\stackrel{\xi}{\to} S_{2-k,\rho}\to 0
\end{equation}

\subsection{Regularized Theta Lifts}
Now we consider the regularized theta integral as a limit of truncated integrals as follows
\begin{equation}
\label{eqn:Green}
\Phi(z,g,f)=\int_{\calF}^{\reg}\langle f(\tau),\theta(\tau,z,g)\rangle\md\mu(\tau)
=\CT_{s=0}\left[\lim_{T\to\infty}\int_{\calF_T}\langle f(\tau),\theta(\tau,z,g)\rangle v^{-s}\md\mu(\tau)\right],
\end{equation}
where $\CT_{s=0}[F]$ means the constant term in the Laurent expansion of the function $F$ at $s=0$ and
$$\calF_T=\{\tau\in\calF\mid\im\tau\leq T\},$$
is the truncated fundamental domain.

This theta lift was first studied by Borcherds \cite{Bor98} for weakly holomorphic modular forms and got its name ``Borcherds lift'' and later Bruinier and Funke \cite{BF04} generalized the lift to make it work on harmonic weak Maass forms and most importantly proved that
\begin{thm}[\cite{BF04}]
Let $f:\H\to S(V(\A_f))^K$ with same notations as above. The function $\Phi(z,g,f)$ is smooth on $X_K\bs Z(f)$, where
$$Z(f)=\sum_{m>0}Z(m,c^+(-m)).$$
It has a logarithmic singularity along the divisor $-2Z(f)$. The $(1,1)$-form $\dd\dd^c\Phi(z,g,f)$ can be continued to a smooth form on all of $X_K$. And we have the Green current equation
$$\dd\dd^c[\Phi(z,g,f)]+\delta_{Z(f)}=[\dd\dd^c\Phi(z,g,f)],$$
where $\delta_Z$ denotes the Dirac current of a divisor $Z$. Moreover, if $\Delta_z$ denotes the invariant Laplacian operator on $\D$, normalized in \cite{Bru02}, we have
$$\Delta_z\Phi(z,g,f)=\frac{n}{4}\cdot c^+(0)(0).$$
\end{thm}
In particular, the theorem implies that $\Phi(z,g,f)$ is a Green function for the divisor $Z(f)$ in the sense of Arakelov geometry in the normalization of \cite{SABK92}. Moreover, we see that $\Phi(z,g,f)$ is harmonic when $c^+(0)(0)=0$. It is often called the automorphic Green function associated with $Z(f)$.

\begin{thm}[\cite{Bru02}]
There exists $f_{m,\varphi}\in H_{1-n/2}(\Schw)$, such that $Z(m,\varphi)=Z(f)$. Moreover, $f_{m,\varphi}$ is unique when $n>2$ or $n=2$ and $V$ is anisotropic.
\end{thm}

\begin{thm}[\cite{Bor98}]
\label{thm:Borcherds}
Assume $f\in M_{1-n/2,\rhobar}^!$ such that $c^+(-m)$ is integral valued for all $m>0$. Then there exists a unique (up to a constant of modulus 1) meromorphic modular form $\Psi(z,g,f)$ on $G=\GSpin(V)$  of weight $c^+(0)(0)$ satisfying
$$\div\Psi(f)=Z(f),\text{ and }-\log ||\Psi(z,g,f)||_{\Pet}^2=\Phi(z,g,f),$$
where
$$||\Psi(z^{\pm},g,f)||_{\Pet}^2=|\Psi(z^{\pm},g,f)|^2(4\pi e^{\Gamma'(1)}y^+y^-)^{c^+(0,0)}$$
is the normalized Petersson metric. Moreover, $\Psi(z,g,f)$ has an infinite converging product expansion near a cusp if any.
\end{thm}

\section{Eisenstein Series}
Let us fix the non-trivial character $\psi$ to be the canonical unramified additive character of $\Q\backslash\A$ with $\psi_{\infty}(x)=e^{2\pi ix}$ for $x\in\R$. Let $\psi_F$ be 
$$\psi_F=\psi\circ\tr_{F/\Q}.$$

Following Section \ref{cha:Weil}, there is a unique Weil representation $\omega=\omega_{\psi_F}$ of $\widetilde{\SL_2(\A_F)}$ on $S(W(\A_F))$ associated to the quadratic space $(W,(,)_W)$ and a non-trivial additive character $\psi_F$. In our case $\dim_F W=2$ is even, it is well-known that the Weil representation actually factors through $\SL_2(\A_F)$.

Let $\chi:F^{\times}\backslash\A_F^{\times}\to\C^{\times}$ be the quadratic Hecke character associated to $E/F$. Let $B$ be the Borel subgroup of $\SL_2$. Then $B$ has a decomposition $B=NM$ that satisfies for any $F$-algebra $R$, we have
$N=\{n(b)\mid b\in R\}$, $M=\{m(a)\mid a\in R^{\times}\}$, where
\begin{eqnarray*}
n(b)=\matrix{1}{b}{0}{1},\quad m(a)=\matrix{a}{0}{0}{a^{-1}}.
\end{eqnarray*}
Then $\chi$ is also the quadratic Hecke character associated to $W$, and there is an $\SL_2(\A_F)$-equivariant map
\begin{equation}
\lambda=\prod\lambda_v:S(W(\A_F))\to I(0,\chi),\quad\lambda(\varphi)(g)=\omega(g)\varphi(0)
\end{equation}
Here $I(s,\chi)=\Ind_{B(A_F)}^{\SL_2(\A_F)}(\chi\cdot|~|^s)$ is the principal series, whose sections are smooth functions $\Phi$ on $\SL_2(\A_F)$ such that
$$\Phi(n(b)m(a)g,s)=\chi(a)|a|^{s+1}\Phi(g,s)$$
for any $a\in\A_F^{\times}$, $b\in\A_F$ and $g\in\SL_2(\A_F)$.

$\Phi$ is called standard if $\Phi|_K$ is independent of $s$, where $K=\SL_2(\hat{\O}_F)\SO_2(\R)^{d+1}$ is the maximal compact open subgroup of $\SL_2(\A_F)$. And it is called factorizable if $\Phi=\otimes\Phi_v$, with $\Phi_v\in I(s,\chi_v)$.

For a standard section $\Phi\in I(s,\chi)$, its associated Eisenstein series is defined as 
$$E_W(g,s,\Phi)=\sum_{\gamma\in B(F)\backslash\SL_2(F)}\Phi(\gamma g,s).$$
For simplicity, we denote $E_W$ simply by $E$ in this section from now on.

By general theory of Eisenstein series, the summation defining $E(g,s,\Phi)$ is absolutely convergent when $\re(s)$ is sufficiently large, and has meromorphic continuation to the whole complex plane with finitely many poles. The meromorphic continuation is holomorphic along $\re(s)=0$ and satisfies a functional equation in $s\mapsto -s$. Furthermore, there is a Fourier expansion 
$$E(g,s,\Phi)=\sum_{t\in F}E_t(g,s,\Phi),$$
where
$$E_t(g,s,\Phi)=\int_{F\backslash\A_F}E(n(b)g,s,\Phi)\cdot\psi_F(-bt)\md b.$$

Here $\mathrm{d}b$ is the Haar measure on $F\backslash\A_F$ self-dual with respect to $\psi_F$. If $\Phi=\otimes\Phi_v$ is factorizable and $t\in F^{\times}$, there is a factorization
$$E_t(g,s,\Phi)=\prod_v W_{t,v}(g_v,s,\Phi_v),$$
where
$$W_{t,v}(g_v,s,\Phi_v)=\int_{F_v}\Phi_v(w^{-1}n(b)g_v,s)\cdot\psi_{F_v}(-bt)\md b,$$
and $w=\matrix{0}{1}{-1}{0}$.

For $\varphi\in S(W_f)$, let $\Phi_f$ be the standard section associated to $\lambda_f(\phi)\in I(0,\chi_f)$. For each real embedding $\sigma_i:F\hookrightarrow\R$, the maximal compact subgroup $K_{\sigma_i}\cong\SO_2(\R)$ is abelian with character $k_{\theta}\mapsto e^{ik\theta}$ indexed by $k\in\Z$. Using the decomposition
$\SL_2(F_{\sigma_i})=B(F_{\sigma_i})\cdot K_{\sigma_i}$ 
and the fact that $\chi$ is odd, it follows that
$$I(s,\chi_{\C/\R})=I(s,\chi_{E_{\sigma_i}/F_{\sigma_i}})=\bigoplus_{k\text{ odd}}\C\Phi_{\sigma_i}^k,$$
where $\Phi_{\sigma_i}^k\in I(s,\chi_{\C/\R})$ is the unique standard section satisfying
$$\Phi_{\sigma_i}^k(n(b)m(a)k_{\theta})=\chi_{\C/\R}(a)|a|^{s+1}e^{ik\theta},$$
for $a\in\R^{\times}$, $b\in\R$, and $k_{\theta}=\matrix{\cos\theta}{\sin\theta}{-\sin\theta}{\cos\theta}\in\SO_2(\R)$. Then for $\vectau=\vec{u}+i\vec{v}\in\H^{d+1}$ and a standard section $\Phi_f\in I(s,\chi_f)$, we define
$$E(\vectau,s,\varphi,\vecid)=\norm_{F/\Q}(v)^{-\frac{1}{2}}E(g_{\vectau},s,\Phi_f\otimes\Phi_{\infty}^{\vecid}),$$
where $\Phi_{\infty}^{\vecid}=\otimes_{i=0}^d\Phi_{\sigma_i}^1$ and $g_{\vectau}=n(\vec{u})m(\vec{v}^{1/2})$ viewed as an element of $\SL_2(\A_F)$ with trivial non-archimedean components.

It is a non-holomorphic Hilbert modular form of parallel weight 1. We further normalize
$E^*(\vectau,s,\varphi,\vecid)=\Lambda(s+1,\chi)E(\vectau,s,\varphi,\vecid),$
where
$$\Lambda(s,\chi)=\norm_{F/\Q}^{\frac{s}{2}}(\partial_F d_{E/F})\left(\pi^{-\frac{s+1}{2}}\Gamma\left(\frac{s+1}{2}\right)\right)^{d+1}L(s,\chi),$$
and $\partial_F$ is the different of $F$, $d_{E/F}$ is the relative discriminant of $E/F$. 

The Eisenstein series is incoherent in the sense that all $\Phi_v$ except $\Phi_{\sigma_0}$ come from some $\lambda(\varphi_v)$. This forces $E^*(\vectau,0,\varphi,\vecid)=0$ automatically.
\begin{prop}[\cite{BKY12}*{Proposition 4.6}]
\label{prop:approx}
Let $\varphi\in S(W(\hat{F}))$. For a totally positive element $t\in F_+^{\times}$, let $a(t,\varphi)$ be the $t$-th Fourier coefficient of $E^{*,\prime}(\vectau,0,\varphi,\vecid)$ and write the constant term of $E^{*,\prime}(\vectau,0,\varphi,\vecid)$ as
$$\varphi(0)\Lambda(0,\chi)\log\norm(\vec{v})+a_0(\varphi).$$
Let $$\E(\tau,\varphi)=a_0(\varphi)+\sum_{m\in\Q_+} a_m(\varphi)q^m,$$
where 
$$a_m(\varphi)=\sum_{t\in F_+^{\times},\tr_{F/Q}t=m}a(t,\varphi).$$
Finally, write $\tau^{\bigtriangleup}=(\tau,\cdots,\tau)$ for the diagonal image of $\tau\in\H$ in $\H^{d+1}$, then
$$E^{*,\prime}(\tau^{\bigtriangleup},0,\varphi,\vecid)-\E(\tau,\varphi)-\varphi(0)\left(d+1\right)\Lambda(0,\chi)\log v$$
is of exponentially decay as $v$ goes to infinity. Moreover,
$$a_n(\varphi)=\sum_p a_{n,p}(\varphi)\log p$$
with $a_{n,p}(\varphi)\in\Q(\varphi)$, the subfield of $\C$ generated by the values $\varphi(x)$, $x\in W(\hat{F})=V(\A_f)$. 
\end{prop}

\begin{lem}[\cite{BKY12}*{Lemma 4.3}]
\label{lem:Eisenstein}
$$-2\partialbar_j(E_W'(\vectau,0,\vecid)\mathrm{d}\tau_{\sigma_j})=E_W(\vectau,0,\vecid(j))\mathrm{d}\mu(\tau_{\sigma_j}).$$
\end{lem}

\begin{prop}[\cite{BKY12}*{Proposition 4.5}]
\label{prop:Siegel-Weil}
$$\theta_W(\tau,Z(T(j),h_0^\pm(j),g_j))=\frac{1}{2}\deg(Z(T,h_0^\pm))\cdot E_W(\tau^{\bigtriangleup},0,\vecid(j)).$$
\end{prop}

\section{Main Theorem}
\subsection{CM Values of Green Functions}
Now we are ready to state and prove our main general formula.
\begin{thm}
\label{thm:Main}
For a $K$-invariant harmonic weak Maass form $f\in H_{1-n/2,\rhobar}$ with $f=f^++f^-$ as in (\ref{eqn:Fourier}) and with notation as above,
$$\Phi(Z(W),f)=\frac{\deg(Z(T,z_0^{\pm}))}{\Lambda(0,\chi)}\left(\CT[\langle f^+,\theta_0\otimes\E_W\rangle]-\L^{*,\prime}_W(0,\xi(f))\right),$$
where
\begin{eqnarray*}
\L_W(s,g)&=&\langle g(\tau),\theta_0\otimes E_W(\tau^{\bigtriangleup},s,\vecid)\rangle_{\Pet},\\
\L^*_W(s,g)&=&\Lambda(s+1,\chi)\L_W(s,g)
\end{eqnarray*}
are a Rankin-Selberg convolution L-function and its completion for $g\in S_{1+n/2,\rho}$.
\end{thm}
\begin{proof}
First, by Lemma \ref{lem:Eisenstein} and Proposition \ref{prop:Siegel-Weil}, we have
\begin{eqnarray*}
\Phi(Z(T(j),z_0(j),g_j),f)
&=&\int_{\calF}^{\reg}\langle f(\tau),\theta(\tau,Z(T(j),z_0(j),g_j))\rangle\md\mu(\tau)\\
&=&\int_{\calF}^{\reg}\langle f(\tau),\theta_0(\tau)\otimes\theta_{W}(\tau,Z(T(j),z_0(j),g_j))\rangle\md\mu(\tau)\\
&=&\frac{1}{2}\deg(Z(T,z_0))\int_{\calF}^{\reg}\langle f(\tau),\theta_0(\tau)\otimes E_W(\tau^{\bigtriangleup},0,\vecid(j))\md\mu(\tau)\rangle\\
&=&-\deg(Z(T,z_0))\int_{\calF}^{\reg}\langle f(\tau),\theta_0(\tau)\otimes \partialbar_j(E'_W(\tau^{\bigtriangleup},0,\vecid)\md\tau)\rangle.
\end{eqnarray*}

Recall the definition \ref{eqn:cycle} of $Z(W)$ as a sum of $Z(T(j),z_0(j),g_j)$, we have
\begin{eqnarray*}
\Phi(Z(W),f)
&=&-2\deg(Z(T,z_0))\int_{\calF}^{\reg}\langle f(\tau),\theta_0(\tau)\otimes\sum_{j=0}^d\partialbar_j(E'_W(\tau^{\bigtriangleup},0,\vecid)\md\tau)\rangle\\
&=&-2\deg(Z(T,z_0))\int_{\calF}^{\reg}\langle f(\tau),\theta_0(\tau)\otimes\partialbar(E'_W(\tau^{\bigtriangleup},0,\vecid)\md\tau)\rangle\\
&=&-2\deg(Z(T,z_0))\int_{\calF}^{\reg}\md(\langle f(\tau),\theta_0(\tau)\otimes E'_W(\tau^{\bigtriangleup},0,\vecid)\md\tau)\rangle)\\
&&\qquad+2\deg(Z(T,z_0))\int_{\calF}^{\reg}\langle \partialbar f(\tau),\theta_0(\tau)\otimes E'_W(\tau^{\bigtriangleup},0,\vecid)\md\tau)\rangle\\
&=&-\frac{2\deg(Z(T,z_0))}{\Lambda(1,\chi)}I_1+\frac{2\deg(Z(T,z_0))}{\Lambda(1,\chi)}I_2,
\end{eqnarray*}
where
\begin{eqnarray*}
I_1&=&\int_{\calF}^{\reg}\md(\langle f(\tau),\theta_0(\tau)\otimes E_W^{*,\prime}(\tau^{\bigtriangleup},0,\vecid)\md\tau\rangle),\\
I_2&=&\int_{\calF}^{\reg}\langle \partialbar f(\tau),\theta_0(\tau)\otimes E_W^{*,\prime}(\tau^{\bigtriangleup},0,\vecid)\md\tau\rangle.
\end{eqnarray*}
Recall that
$$\partialbar f(\tau)=-\frac{1}{2i}v^{\frac{n}{2}-1}\overline{\xi(f)}\md\bar{\tau}.$$
Thus,
$$\langle \partialbar f(\tau),\theta_0(\tau)\otimes E_W^{*,\prime}(\tau^{\bigtriangleup},0,\vecid)\md\tau\rangle
=-\langle\overline{\xi(f)},\theta_0(\tau)\otimes E_W^{*,\prime}(\tau^{\bigtriangleup},0,\vecid)\rangle v^{\frac{n}{2}+1}\md\mu(\tau)$$
is integrable over the fundamental domain $\calF$, and hence
$$I_2=-\int_{\calF}\langle\overline{\xi(f)},\theta_0(\tau)\otimes E_W^{*,\prime}(\tau^{\bigtriangleup},0,\vecid)\rangle v^{\frac{n}{2}+1}\md\mu(\tau)=-\L^{*,\prime}_W(0,\xi(f)).$$

By a similar argument in \cite{Kud03}, there is a unique constant $A_0$ such that
\begin{eqnarray*}
I_1&=&\lim_{T\to\infty}\left(\int_{\calF_T}\md(\langle f(\tau),\theta_0(\tau)\otimes E_W^{*,\prime}(\tau^{\bigtriangleup},0,\vecid)\md\tau)\rangle)-A_0\log T\right)\\
&=&\lim_{T\to\infty}(I_1(T)-A_0\log T).
\end{eqnarray*}
By Stokes' theorem, one has
\begin{eqnarray*}
I_1(T)
&=&\int_{\partial\calF_T}\langle f(\tau),\theta_0(\tau)\otimes E_W^{*,\prime}(\tau^{\bigtriangleup},0,\vecid)\md\tau)\rangle\\
&=&-\int_{iT}^{iT+1}\langle f(\tau), \theta_0(\tau)\otimes E_W^{*,\prime}(\tau^{\bigtriangleup},0,\vecid)\rangle\md\tau\\
&=&-\int_{iT}^{iT+1}\langle f^+(\tau), \theta_0(\tau)\otimes E_W^{*,\prime}(\tau^{\bigtriangleup},0,\vecid)\rangle\md\tau+O(e^{-\epsilon T})
\end{eqnarray*}
for some $\epsilon>0$ since $f^-$ is of exponential decay and $E_W^{*,\prime}$ is of moderate growth. Proposition \ref{prop:approx} asserts that 
$$E_W^{*,\prime}(\tau^{\bigtriangleup},0,\vecid)=\mathcal{E}_W(\tau)+\Lambda(0,\chi)(d+1)\log v+\sum_{m\in\Q_{>0}}a(m,v)q^m$$
such that $a(m,v)q^m$ is of exponential decay as $v\to\infty$. Thus,
$$-I_1(T)=\mathrm{CT}[\langle f^+(\tau),\theta_0\otimes\mathcal{E}_W(\tau)\rangle]+\Lambda(0,\chi)(d+1)\log T+\sum_{m\in\Q_{>0}}c^+(-m)a(m,T).$$
The last sum goes to zero when $T\to\infty$. So we can take $A_0=(d+1)\Lambda(0,\chi),$ and 
$$I_1=-\mathrm{CT}[\langle f^+(\tau),\theta_0\otimes\mathcal{E}_W(\tau)\rangle]$$
as claimed.
\end{proof}

\subsection{Lattice Version}
\label{sec:lattice}
Let $L$ be an even integral lattice in $V$, i.e. $Q(x)=\frac{1}{2}(x,x)\in\Z$ for $x\in L$, and let
$$L'=\{y\in V\mid (x,y)\in\Z\text{ for every }x\in L\}\supset L$$
be its dual.

For $\mu\in L'/L$, we write $\varphi_{\mu,L}=\mathrm{char}(\mu+\hat{L})\in\Schw$ and $Z(m,\mu)=Z(m,\varphi_{\mu,L})$, where $\hat{L}=L\otimes\Zhat$. Recall from Section \ref{cha:Weil}, there is a Weil representation $\omega=\omega_{\psi}$ of $\widetilde{\SL_2(\A)}$ on $\Schw$ with non-trivial additive character $\psi$. Denote the subspace $\oplus\C\varphi_{\mu,L}\subset\Schw$ by $S_L$.

Since the subspace $S_L$ is preserved by the action of $\widetilde{\SL_2(\Zhat)}$, there is a representation $\rho_L$ of $\widetilde{\SL_2(\Z)}$ on this space defined by the formula
$$\rho_L(\tgamma)\varphi=\bar{\omega}_f(\hat{\gamma})\varphi,$$
where $\tgamma=(\gamma,\epsilon)\in\widetilde{\SL_2(\Z)}\subset\widetilde{\SL_2(\R)}$, $\eta(\gamma)\in\widetilde{\SL_2(\A)}$ under the splitting (\ref{eqn:splitting}) can be written as $\hat{\gamma}\tgamma$, where $\hat{\gamma}=(\gamma,\hat{\epsilon})$ is the finite ad\`eles part and $\tgamma$ is the infinity ad\`eles part. This representation is given explicitly by Borcherds \cite{Bor98} as
\begin{eqnarray}
\rho_L(T)(\varphi_{\mu,L})&=&e(Q(\mu^2))\varphi_{\mu,L},\\
\rho_L(S)(\varphi_{\mu,L})&=&\frac{e((2-n)/8)}{\sqrt{|L'/L|}}\sum_{\nu\in L'/L}e(-(\mu,\nu))\varphi_{\nu,L},
\end{eqnarray}
where $T=\left(\matrix{1}{1}{0}{1},1\right)$ and $S=\left(\matrix{0}{-1}{1}{0},\sqrt{\tau}\right)$ and $\widetilde{\SL_2(\Z)}$ is generated by $T$ and $S$. Note that the complex conjugate $\bar{\rho}_L$ is thus the restriction of $\omega$ to the subgroup $\SL_2(\Z)\subset\SL_2(\Zhat)$.

Then the Fourier expansion of any $f\in H_{k,\rho_L}$ gives a unique decomposition $f=f^++f^-$, where
\begin{equation}
f^+=\sum_{\mu\in L'/L}f^+_{\mu,L}\varphi_{\mu,L}.
\end{equation}

Using the splitting (\ref{eqn:decomposition}), we obtain definite lattices
$$L_0=L\cap V_0,\quad M=L\cap\Res_{F/\Q}W.$$
Then $L_0\oplus M\subset L$ is a sub-lattice of finite index. 

$$\theta_0=\sum_{\mu_0\in L_0'/L_0}\theta_0(\varphi_{\mu_0,L_0})\varphi_{\mu_0,L_0}^{\vee},$$
$$\E_W=\sum_{\mu_1\in M'/M}\E_W(\varphi_{\mu_1,M})\varphi_{\mu_1,M}^{\vee}.$$

It is easy to see that $L_0\oplus M\subset L\subset L'\subset L_0'\oplus M'$ and therefore,
$$L'/(L_0\oplus M)\subset (L_0'\oplus M')/(L_0\oplus M),$$
and there is a surjection $\varpi:L'/(L_0\oplus M)\to L'/L$. 

As a result, for $\mu\in L'/L$, we have 
$$\varpi^{-1}(\mu)\subset L'/(L_0\oplus M)\subset (L_0'\oplus M')/(L_0\oplus M)$$ 
and hence
$$\varphi_{\mu,L}=\sum_{(\mu_0,\mu_1)\in\varpi^{-1}(\mu)}\varphi_{\mu_0,L_0}\otimes\varphi_{\mu_1,M}.$$

\begin{eqnarray}
\langle f^+,\theta_0\otimes\E_W\rangle
&=&\sum_{(\mu_0,\mu_1)\in\varpi^{-1}(\mu)}f^+_{\mu,L}\theta_0(\varphi_{\mu_0,L_0})\E_W(\varphi_{\mu_1,M})
\end{eqnarray}

\begin{thm}[Lattice Version of the Main Theorem]
\label{thm:lattice}
For a harmonic weak Maass form $f\in H_{1-n/2,\rhobar}$ valued in $S_L$ with $f=f^++f^-$ as in (\ref{eqn:Fourier}) and with notation as above,
$$\Phi(Z(W),f)=\frac{\deg(Z(T,z_0^{\pm}))}{\Lambda(0,\chi)}\left(\CT\left[\sum_{\mu\in L'/L}\sum_{(\mu_0,\mu_1)\in\varpi^{-1}(\mu)}f^+_{\mu,L}\theta_0(\varphi_{\mu_0,L_0})\E_W(\varphi_{\mu_1,M})\right]-\L^{*,\prime}_W(0,\xi(f))\right).$$
\end{thm}

\section{Siegel 3-Fold}
Now for a concrete example, let us apply our main theorem to the special case of Siegel 3-fold.

\subsection{Classical Definition}
First, we can define Siegel upper half plane 
$$\H_2=\{\tau\in M_{2\times 2}(\C)\mid \tau\text{ symmetric and }\im(\tau)\text{ positive definite}\}$$
and symplectic group 
$$\Sp_4(\Q)=\left\{\left.g=\matrix{A}{B}{C}{D}\in\GL_4(\Q)\right|~A^tD-C^tB=I_2,~A^tC=C^tA,~B^tD=D^tB\right\}.$$

$\Sp_4(\Q)$ acts on $\H_2$ by fractional linear transformation
$$g\cdot\tau:=(A\tau+B)(C\tau+D)^{-1}.$$

Similar to the modular curve case, we can define congruence subgroups $\Gamma_2(N)$ of $\Sp_4(\R)$ as follows
$$\Gamma_2(N)=\ker(\Sp_4(\Z)\to\Sp_4(\Z/N\Z)).$$

We are particularly interested in the following quotient space, or Siegel 3-fold
$$X_2(2)=\Gamma_2(2)\backslash\H_2.$$
It is the moduli space of the triples $(A,\lambda,\psi:A[2]\stackrel{\sim}{\to}(\Z/2\Z)^4)$. Here $A$ is a principally polarized abelian scheme of relative dimension 2 with polarization $\lambda$. $\psi$ is an isomorphism preserving the symplectic forms between the Weil pairing on $A[2]\times A^{\vee}[2]$ and the standard symplectic pairing on $(\Z/2\Z)^4$.

\subsection{As an Orthogonal Shimura Variety}
\subsubsection{Realization}
\label{sec:realization} 
In order to identify Siegel 3-fold as an orthogonal Shimura variety defined in Section \ref{cha:SV}, we take the following $V,~W_0,~V_0$ with associated lattices $L,~M,~L_0$ in (\ref{eqn:decomposition}) as follows.

Let 
$$V=\left.\left\{A=\left(\begin{array}{cccc}
r&-c&0&-a\\
d&-r&a&0\\
0&-b&r&d\\
b&0&-c&-r
\end{array}\right)\right|a,b,c,d,r\in\Q\right\}$$
with quadratic form 
$$Q_V(A)=\frac{1}{2}\tr(A^2)=2r^2-2ab-2cd$$ 
and lattice $L={(a,b,c,d,r)\in\Z^5}\subset V$. Therefore, $L'/L\cong\left(\dfrac{1}{2}\Z\Big/\Z\right)^4\oplus\left(\dfrac{1}{4}\Z\Big/\Z\right)$.

Let 
$$W_0=\left\{\left(\begin{array}{cccc}
r&-Ds&0&-a\\
s&-r&a&0\\
0&-b&r&s\\
b&0&-Ds&-r
\end{array}\right)\right\}$$
with quadratic form $Q_{W_0}=2r^2-2ab-2Ds^2$ and lattice $M={(a,b,r,s)\in\Z^4}\subset W_0$. Therefore, $M'/M\cong\left(\dfrac{1}{2}\Z\Big/\Z\right)^2\oplus\left(\dfrac{1}{4}\Z\Big/\Z\right)\oplus\left(\dfrac{1}{4D}\Z\Big/\Z\right)$.

Let 
$$V_0=\left\{\left(\begin{array}{cccc}
0&Dt&0&0\\
t&0&0&0\\
0&0&0&t\\
0&0&Dt&0
\end{array}\right)\right\}$$
with quadratic form $Q_{V_0}=2Dt^2$ and lattice $L_0={(t)\in\Z}\subset V_0$. Therefore, $L_0'/L_0\cong\dfrac{1}{4D}\Z\Big/\Z$.

Let
$$W=\left\{\left.\matrix{\sigma_1(u)}{a}{b}{\sigma_2(u)}\right| u\in F,~a,b\in\Q\right\}$$
with quadratic form $Q_W=2\norm_{\Sigma}(u)-2ab$. It is easy to see that $W_0\cong\Res_{F/\Q}W$.

\subsubsection{Identification}
\label{sec:identification}
First, let us identify $\GSp_4$ with $G=\GSpin(V)$ in Section \ref{cha:SV}. 

It is easy to check that $\GSp_4(\Q)$ acts on $V$ via $\Ad_g(X) =gAg^{-1}$ for any $g\in\GSp_4(\Q)$ and $A\in V$. Then, by direct computation
$$Q_V(gAg^{-1})=\frac{1}{2}\tr((gAg^{-1})^2)
=\frac{1}{2}\tr(gA^2g^{-1})
=\frac{1}{2}\tr(A^2)
=Q_V(A),$$
we can see that this adjoint action of $g$ on $V$ also preserves the quadratic form $Q_V$. Hence it gives the well-known identification $\GSp_4\cong\GSpin(V)$ with
$$1\to\G_m\to\GSp_4\to\SO(V)\to 1.$$
Under this identification, one also has $\Sp_4\cong\Spin(V)$.

Next, let us identify $\H_2$ with $\D^+$ in Section \ref{cha:SV}.

Let us construct the following space
\begin{eqnarray*}
\mathcal{D}&=&\{\text{negative lines in }V(\C)=V\otimes_{\Q}\C\}\\
&=&\{z\in V(\C)\mid Q_V(z)=0,~B(z,\bar{z})<0\}/\C^{\times}\\
&=&\{[a,b,c,d,r]\in V(\C)\mid r^2=ab+cd,~2|r|^2-\bar{a}b-a\bar{b}-\bar{c}d-c\bar{d}<0\}/\C^{\times},
\end{eqnarray*}
which is a complex manifold of dimension 3 consisting of 2 connected components. Here 
$$B(x,y)=Q_V(x+y)-Q_V(x)-Q_V(y)$$ is the bilinear form associated to $Q_V$. It is clear that $\mathcal{D}$ is well-defined using basic properties of bilinear form $B$ and quadratic form $Q_V$.

Now we would like to identify both $H_2^{\pm}$ and $\D$ with $\mathcal{D}$ by the following two lemmas.

\begin{lem}
There is an identification 
\begin{eqnarray*}
\mathcal{D}&\stackrel{\cong}{\rightarrow}&\D\\
z=x+iy&\mapsto&\R x+\R y,
\end{eqnarray*}
where $x=\re(z),~y=\im(z)\in V(\R)$.
\end{lem}


\begin{lem}
There is an identification $\Xi:\H_2^\pm\stackrel{\cong}{\rightarrow}\mathcal{D}$ given by
$$\quad\tau=\matrix{\tau_1}{\tau_{12}}{\tau_{12}}{\tau_2}\mapsto\matrix{(J\tau)^t}{J\det(\tau)}{-J}{J\tau},$$
where $J=\matrix{0}{1}{-1}{0}$, $\H_2^+=\H_2$, and $\H_2^-$ consists of symmetric complex matrices $\tau$ of size $2\times 2$ such that $\im(\tau)$ is negative definite. 

Conversely, for $z\in\mathcal{D}$ who has coordinate $[a,b,c,d,r]$, it can be proved that $ab\not=0$. Then $\Xi$ clearly has an inverse map
\begin{eqnarray*}
\Xi^{-1}:\mathcal{D}&\stackrel{\cong}{\rightarrow}&\H_2^{\pm}\\
z=[a,b,c,d,r]&\mapsto&\frac{1}{b}\matrix{c}{r}{r}{d}
\end{eqnarray*}
\end{lem}

\begin{proof}
According to our construction, $\Xi(\tau)$ has coordinate
$$[-\det(\tau),1,\tau_1,\tau_2,\tau_{12}].$$
Hence, we have
\begin{eqnarray*}
Q_V(\Xi(\tau))&=&2\tau_{12}^2-2\tau_1\tau_2-2(-\det(\tau))=0,\\
B(\Xi(\tau),\overline{\Xi(\tau)})
&=&2|\tau_{12}|^2-\tau_1\overline{\tau_2}-\overline{\tau_1}\tau_2+\det(\tau)+\det(\bar{\tau})\\
&=&2|\tau_{12}|^2-\tau_1\overline{\tau_2}-\overline{\tau_1}\tau_2+\tau_1\tau_2-\tau_{12}^2+\overline{\tau_1}\overline{\tau_2}-\overline{\tau_{12}}^2\\
&=&(\tau_1-\overline{\tau_1})(\tau_2-\overline{\tau_2})-(\tau_{12}-\overline{\tau_{12}})^2=-4\det(\im(\tau))<0.
\end{eqnarray*}
Therefore, $\Xi(\tau)\in\mathcal{D}$.

For the claim about $ab\not=0$ for $[a,b,c,d,r]\in\mathcal{D}$. If we assume that $ab=0$, we will have both $r^2=cd$ and $r^2<\re(\bar{c}d)$, which causes a contradiction.

Conversely, the definition of $\Xi^{-1}$ clearly does not depend on the choices of $[a,b,c,d,r]$. Now let us verify that the image does lie in $\H_2^{\pm}$.
\begin{eqnarray*}
-4\det(\Xi^{-1}(z))&=&\left(\frac{c}{b}-\frac{\bar{c}}{\bar{b}}\right)\left(\frac{d}{b}-\frac{\bar{d}}{\bar{b}}\right)-\left(\frac{r}{b}-\frac{\bar{r}}{\bar{b}}\right)^2\\
&=&\frac{(c\bar{b}-\bar{c}b)(d\bar{b}-\bar{d}b)-(r\bar{b}-\bar{r}b)^2}{|b|^4}\\
&=&\frac{(cd-r^2)\bar{b}^2+(\bar{c}\bar{d}-\bar{r}^2)b^2-(c\bar{d}+\bar{c}d-2|r|^2)|b|^2}{|b|^4}\\
&=&\frac{-ab\bar{b}^2-\bar{a}\bar{b}b^2-(c\bar{d}+\bar{c}d-2|r|^2)|b|^2}{|b|^4}\\
&=&\frac{-a\bar{b}-\bar{a}b-c\bar{d}-\bar{c}d+2|r|^2}{|b|^2}<0
\end{eqnarray*}
Since $\Xi^{-1}(z)$ is already symmetric of size $2\times 2$, it is either positive definite or negative definite.
\end{proof}

\subsection{Embeddings of Hilbert Modular Surfaces}
\label{sec:relation}
In our setup, it is very easy to see that $W_0=\Res_{F/\Q}W$ is the underlying vector space of orthogonal Shimura variety of signature (2,2). And it is actually related to Hilbert modular surfaces. 

Let $F=\Q(\sqrt{D})$ be a real quadratic field with fundamental discriminant $D$. Denote the ring of integers of $F$ by $\O_F$, and its different by $\partial_F=\sqrt{D}\O_F$. For convenience, we fix the following $\Z$-basis $\{e_1,e_2\}$ of $\O_F$ with $e_1=1$ and
\begin{equation}
e_2=\left\{\begin{array}{ll}
\dfrac{1-\sqrt{D}}{2}&\text{ if }D\equiv 1~(\mod~4),\\
\dfrac{-\sqrt{D}}{2}&\text{ if }D\equiv 0~(\mod~4).
\end{array}
\right.
\end{equation}

Let $\sigma$ be the non-trivial Galois automorphism of $F$ and denote by $\tr_{F/\Q}(t)=t+\sigma(t)$ the standard field trace of $t\in F$ over $\Q$. For $z=(z_1,z_2)\in\H^2$ and $t\in F$, define $tz=(tz_1,\sigma(t)z_2)$. We also define the trace $\tr_{F/\Q}$ on $\H^2$ as the map $\tr_{F/\Q}(z)=z_1+z_2$. Hence, for $t\in F$, we have $\tr_{F/\Q}(tz)=tz_1+\sigma(t)z_2$.

The group $\SL_2(F)$ acts on $\H^2$ via $\gamma\cdot z=(\gamma\cdot z_1,\sigma(\gamma)\cdot z_2)$, where
$$\gamma\cdot z=\frac{az+b}{cz+d},\text{ if }\gamma=\matrix{a}{b}{c}{d},$$
and $\sigma(\gamma)$ is the matrix obtained by applying $\sigma$ to all the entries in $\gamma$. Let 
$$\SL_2(\O_F\oplus\partial_F^{-1})=\left.\left\{\gamma=\matrix{a}{b}{c}{d}\in\SL_2(F)\right| a,d\in\O_F,~b\in\partial_F^{-1},~c\in\partial_F\right\}.$$

Let
$$R=\matrix{e_1}{e_2}{\sigma(e_1)}{\sigma(e_2)},\quad
\gamma^*=\matrix{a}{b}{c}{d}^*=
\left(\begin{array}{cccc}
a&0&b&0\\
0&\sigma(a)&0&\sigma(b)\\
c&0&d&0\\
0&\sigma(c)&0&\sigma(d)
\end{array}\right)$$

By abuse of language, we define the following three maps using the same notation $\phi$. Let
$$\begin{array}{lrl}
\phi:\H^2\to\H_2,&z=(z_1,z_2)&\mapsto~R^t\diag(z_1,z_2)R,\\
\phi:\SL_2(F)\to\Sp_4(\Q),&\gamma&\mapsto~\diag(R^t,R^{-1})\gamma^*\diag((R^t)^{-1},R),\\
\phi:(\O_F/2\O_F)^2\to(\Z/2Z)^4,&(a,b)&\mapsto~(a_1,a_2,b_1,b_2),
\end{array}$$
where 
$$(a_1,a_2)^t=R^{-1}(a,\sigma(a))^t,\quad (b_1,b_2)=\matrix{\sigma(e_2)}{-\sigma(e_1)}{e_2}{-e_1}^{-1}(b,\sigma(b))^t. $$

We have the following lemma.
\begin{lem}[\cite{LNY16}*{Lemma 2.1}]
\label{lem:relation}
Let the notation be as above. Then
\begin{enumerate}
	\item For $\gamma\in\SL_2(F)$ and $z\in\H^2$, one has $\phi(\gamma\cdot z)=\phi(\gamma)\cdot\phi(z);$
	\item $\phi^{-1}(\Sp_4(\Z))=\SL_2(\O_F\oplus\partial_F^{-1})$;
	\item The map $\phi$ is a bijection between $(\O_F/2\O_F)^2$ and $(\Z/2\Z)^4$ such that $(\va,\vb)$ is even if and only if 
	$\phi(\va,\vb)$ is even.
\end{enumerate}
\end{lem}
In particular, we could also know that
$$\phi^{-1}(\Gamma_2(2))=\widetilde{\Gamma}_2(2)=\left.\left\{\gamma=\matrix{a}{b}{c}{d}\right| a,d\in\O_F,~b\in\partial_F^{-1},~c\in\partial_F,~\gamma\equiv I_2~\mod~2\O_F\right\}.$$

It is also not hard to show that there is a commutative diagram as follows. 
$$\xymatrix{
	\SL_2(\O_F\oplus\partial_F^{-1})\backslash\H^2\ar[d]\ar[r]&\Sp_4(\Z)\backslash\H_2\ar[d]\\
	\tGamma_2(2)\backslash\H^2\ar[r]&\Gamma_2(2)\backslash\H_2
}$$
where $\tGamma_2(2)=\left\{\left.\matrix{a}{b}{c}{d}\in\SL_2(F)\right|a-1,d-1\in 2\O_F,~b\in 2\partial_F^{-1},~c\in 2\partial_F\right\}$.

\subsection{CM Points}
\label{sec:CM}
Let $(E,\Sigma)$ be a quartic CM field with totally real subfield 
$F=\Q(\sqrt{D})$ and CM type $\Sigma=\{\sigma_1,\sigma_2\}$, where $D$ is the fundamental discriminant of $F$. Denote the ring of integers of $F$ by $\O_F$, and its different by $\partial_F=\sqrt{D}\O_F$. Let $\tildeE$ be the reflex field of $(E,\Sigma)$, the subfield of $\C$ generated by the type norm $\norm_{\Sigma}(z)=\sigma_1(z)\sigma_2(z)$, $z\in E$. Then $\tildeE$ is also a quartic CM number field with real subfield $\tildeF=\Q(\sqrt{\tildeD})$ if the absolute discriminant of $E$ is $d_E=D^2\tildeD$. Note that $\tildeD$ is not the fundamental discriminant of $\tildeF$.

Let $\CM_2^{\Sigma}(E)$ be the set of isomorphic classes of principally polarized CM abelian schemes $\boldA=(A,\kappa,\lambda,\psi:A[2]\stackrel{\sim}{\to}(\Z/2\Z)^4)$ of relative dimension 2 over $\C$ of CM type $(\O_E,\Sigma)$ with abelian scheme $A$ over $\C$ with 2-torsion $A[2]$, an $\O_E$-action $\kappa:\O_E\hookrightarrow\End(A)$ and a principally polarization $\lambda:A\to A^{\vee}$ satisfying the further conditions: 
\begin{enumerate}
	\item The Rosati involution induced by $\lambda$ induces the complex conjugation on $E$.
	\item There are two translation invariants, non-zero differentials $\omega_1$ and $\omega_2$ on $A$ over $\C$ such that $\kappa(r)^*\omega_i=\sigma_i(r)\omega_i$ for $r\in\O_E$.
	\item $\psi$ preserves the symplectic forms between the Weil pairing on $A[2]\times A^{\vee}[2]$ and the standard symplectic pairing on $(\Z/2Z)^4$.
\end{enumerate} 

It is known that $X_2(2)$ parametrizes principally polarized abelian schemes $\boldA=(A,\lambda,\psi:A[2]\stackrel{\sim}{\to}(\Z/2\Z)^4)$ of relative dimension 2, where $\psi$ preserves the symplectic forms between the Weil pairing on $A[2]\times A^{\vee}[2]$ and the standard symplectic pairing on $(\Z/2Z)^4$. In other words, there is a map 
\begin{equation}
\label{eqn:CMP}
j:\CM_2(E)=\coprod_{\Sigma}\CM_2^{\Sigma}(E)\to X_2(2).
\end{equation}
which defines a CM point on $X_2(2)$.

Let $\a=H_1(A,\Z)$ with the induced $\O_E$-action and the non-degenerate symplectic form $\lambda:\a\times\a\to\Z$ induced from the polarization of $A$. In particular, $\lambda$ defines a pairing on $\a$ satisfying
$$\lambda(\kappa(r)x,y)=\lambda(x,\kappa(\bar{r})y),\quad r\in\O_E,~x,y\in\a,$$
so that $\a$ is a projective $\O_E$-module of rank one, which is nothing but a fractional ideal $\a$ of $E$. The polarization $\lambda$ induces a polarization $\lambda_{\xi}$ on $\a$ given by
$$\lambda_{\xi}:\a\times\a\to\Z,\quad\lambda_{\xi}(x,y)=\tr_{E/\Q}\xi\bar{x}y,$$
where $\xi\in E^{\times}$ with $\bar{\xi}=-\xi$. A simple calculation shows that $\lambda$ is principally polarized if and only if
\begin{equation}
\label{eqn:condition}
\xi\partial_{E/F}\a\bar{\a}\cap F=\partial_F^{-1}.
\end{equation}
Moreover, $\boldA$ is of CM type $\Sigma$ if and only if $\Sigma(\xi)=(\sigma_1(\xi),\sigma_2(\xi))\in\H^2$. And it is obvious that $\psi:A[2]\stackrel{\sim}{\to}(\Z/2\Z)^4$ will induce a map $\left(\dfrac{1}{2}\a\Big/\a\right)\to(\Z/2\Z)^4$ that preserves the symplectic forms, or equivalently give a symplectic basis of $\left(\dfrac{1}{2}\a\Big/\a\right)$ with respect to Weil pairing.

The converse is also true. Given $(\a,\xi,\underline{e})$ satisfying (\ref{eqn:condition}), where $\underline{e}$ is a symplectic basis of $\left(\dfrac{1}{2}\a\Big/\a\right)$ with respect to Weil pairing, there is a unique CM type $\Sigma$ of $E$ such that $\Sigma(\xi)\in\H^2$ and one has
$$\boldA(\a,\xi,\underline{e}):=(A=(\a\otimes 1)\backslash(E\otimes_{\Q}\R),\kappa,\lambda_{\xi},\underline{e})\in\CM_2^{\Sigma}(E).$$
Here we identify $E\otimes_{\Q}\R$ with $\C^2$ via CM type $\Sigma$.

\subsubsection{Identification of \texorpdfstring{$W$}{Lg} with \texorpdfstring{$\tildeE$}{Lg}}
Let $F=\Q(\sqrt{D})$ be a real quadratic field with fundamental discriminant $D$ and $E=F(\sqrt{\Delta})$ be a totally imaginary quadratic extension of $F$ with CM type $\Sigma$. Let $\tildeF=\Q\sqrt{\Delta\Delta'}$, where $\sigma(r+s\sqrt{D})=(r+s\sqrt{D})'=r-s\sqrt{D}$ is the non-trivial automorphism of $F$ over $\Q$. Then $\tildeE=\Q(\sqrt{\Delta}+\sqrt{\Delta'})$.

For any $(\alpha,\beta)\in E^2$, we define a map $\kappa_{\alpha,\beta}:W_0\to\tildeE,$
\begin{eqnarray}
\kappa(A)&=&a\sigma_1(\alpha)\sigma_2(\alpha)+\sigma_1(\alpha)\sigma_2(\beta)(r-s\sqrt{D})+\sigma_1(\beta)\sigma_2(\alpha)(r+s\sqrt{D})+b\sigma_1(\beta)\sigma_2(\beta)\nonumber\\
&=&a\alpha\sigma(\alpha)+\alpha\sigma(\beta)(r-s\sqrt{D})+\beta\sigma(\alpha)(r+s\sqrt{D})+b\beta\sigma(\beta).
\end{eqnarray}

For a CM point $[\a,\xi,\underline{e}]\in\CM_2^{\Sigma}(E)$, we write
$$\a=\O_F\alpha+\partial_F^{-1}\beta,\quad z=\frac{\beta}{\alpha}\in E^+$$

We define the $\Q$-quadratic form on $\tildeE$ via
$$Q(z)=\tr_{\tildeF/\Q}\frac{1}{\sqrt{\tildeD}}z\bar{z}=\frac{1}{\sqrt{\tildeD}}(z\bar{z}-\sigma(z)\overline{\sigma(z)}).$$

\begin{lem}[\cite{BY06}*{Lemma 4.2}]
\label{lem:norm}
Let $\f_0$ be an integral ideal of $F$, and let $\a=\O_F\alpha+\f_0\beta$ be a fractional ideal of $E$. Then
$$\sqrt{\tildeD}\norm_{E/\Q}\a=\pm 4(\alpha\bar{\beta}-\bar{\alpha}\beta)(\sigma(\alpha)\overline{\sigma(\beta)}-\overline{\sigma(\alpha)}\sigma(\beta))\norm_{F/\Q}\f_0.$$
\end{lem}

\begin{prop}
Let $[\a,\xi,\underline{e}]\in\CM_2^{\Sigma}(E)$, and write $\a=\O_F\alpha+\partial_F^{-1}\beta$. Then the map $\kappa_{\alpha,\beta}$ is a $\Q$-isometry between quadratic spaces $(W_0,Q_{W_0})$ and $(\tildeE,\frac{2\norm_{E/\Q}\a}{\norm_{F/\Q}\f_0}Q)$ with
$$Q(\kappa_{\alpha,\beta}(A))=\frac{\norm_{F/\Q}\f_0}{2\norm_{E/\Q}\a}Q_{W_0}(A).$$
\end{prop}
\begin{proof}
For simplicity, we write $\lambda=r+s\sqrt{D}$ and assume $\beta=1$ and write $z=\alpha$ and
$$\kappa_z(A)=\kappa_{z,1}(A)=az\sigma(z)+z\lambda'+\sigma(z)\lambda+b.$$

When $K$ is cyclic over $\Q$, $\tildeE=E$, then $\kappa_z$ is clearly a $\Q$-linear map. When $K$ is non-Galois, $\tildeE$ is the subfield of $M$ fixed by $\tau$ and thus belongs to $\tildeE$. So $\kappa_z$ is again a $\Q$-linear map. 

Next, we claim that $\rho_z$ is injective. If $\kappa_z(A)=0$, so is $\sigma(\kappa_z(A))$, and thus
$$0=\kappa_z(A)-\sigma(\kappa_z(A))=(a\sigma(z)+\lambda')(z-\bar{z}).$$
Since $z\not\in F$, this identity can only happen if $a=\lambda'=0$ or simply $A=0$. This completes our claim. Notice that $\dim W_0=\dim\tildeE=4$, thus $\kappa_z$ is an isomorphism. 

To check isometry, set $\kappa=\kappa_z(A)$, and it is easy to verify that
\begin{eqnarray*}
\kappa-\sigma(\kappa)&=&(a\sigma(z)+\lambda')(z-\bar{z}),\\
\kappa-\overline{\sigma(\kappa)}&=&(az+\lambda)(\sigma(z)-\overline{\sigma(z)})\\
\bar{\kappa}\sigma(z)-\sigma(\kappa)\overline{\sigma(z)}&=&(\lambda'\bar{z}+b)(\sigma(z)-\overline{\sigma(z)}).
\end{eqnarray*}
So
\begin{eqnarray*}
\kappa\bar{\kappa}-\sigma(\kappa)\overline{\sigma(\kappa)}
&=&\bar{\kappa}(\kappa-\sigma(\kappa))+\sigma(\kappa)(\overline{\kappa-\sigma(\kappa)})\\
&=&(z-\bar{z})\left(a(\bar{\kappa}\sigma(z)-\sigma(\kappa)\overline{\sigma(z)})+\lambda'\overline{\kappa-\overline{\sigma{\kappa}}}\right)\\
&=&(z-\bar{z})(\sigma(z)-\overline{\sigma(z)})(ab-\lambda\lambda')\\
&=&(z-\bar{z})(\sigma(z)-\overline{\sigma(z)})(-\frac{1}{2}Q_{W_0}(A)).
\end{eqnarray*}
By Lemma \ref{lem:norm}, we have
$$Q(\kappa_z(A))=\frac{\norm_{F/\Q}\f_0}{2\norm_{E/\Q}\a}Q_{W_0}(A).$$
\end{proof}

The proposition below determines the image of the lattice $M$ of $W_0$ in $\tildeE$. 

\begin{prop}
\label{prop:lattice}
Let the notation be as above. Then we have
\begin{eqnarray*}
\kappa_{\alpha,\beta}(M)&=&\Z\alpha\sigma(\alpha)+\Z\left(\alpha\sigma(\beta)+\sigma(\alpha)\beta\right)+2\sqrt{D}\Z\left(-\alpha\sigma(\beta)+\sigma(\alpha)\beta\right)+\Z\beta\sigma(\beta),\\
\kappa_{\alpha,\beta}(M')&=&\frac{1}{2}\Z\alpha\sigma(\alpha)+\frac{1}{4}\Z\left(\alpha\sigma(\beta)+\sigma(\alpha)\beta\right)+\frac{1}{2\sqrt{D}}\Z\left(-\alpha\sigma(\beta)+\sigma(\alpha)\beta\right)+\frac{1}{2}\Z\beta\sigma(\beta).
\end{eqnarray*}
Therefore , 
$$\kappa_{\alpha,\beta}(M')/\kappa_{\alpha,\beta}(M)\cong\left(\frac{1}{2}\Z\Big/\Z\right)^2\oplus\left(\frac{1}{4}\Z\Big/\Z\right)\oplus\left(\frac{1}{4D}\Z\Big/\Z\right)\cong M'/M.$$
Furthermore, $\kappa_{\alpha,\beta}(M)$ is of index 2 in $\norm_{\Phi}\a$.
\end{prop} 
In particular, we would like to point out that $M$ and $M'$ are not $\O_F$-lattices.

\subsubsection{Interpretation of \texorpdfstring{$z_0$}{Lg}}
\label{sec:geom}
Now let us review the construction of CM points on $X_2(2)$ of CM type $(\O_E,\Sigma)$. Recall the discussion at the beginning of this section, $\CM_2^{\Sigma}(E)$ can be indexed by the equivalence classes $[\a,\xi,\underline{e}]$, where $\xi\in E^{\times}$ with $\bar{\xi}=-\xi$ and $\a$ is a fractional ideal of $E$ satisfying
$$\xi\partial_{E/F}\a\bar{\a}\cap F=\partial_F^{-1},$$
and $\underline{e}$ is a symplectic basis of $\left(\dfrac{1}{2}\a\Big/\a\right)$ with respect to Weil pairing.

Two pairs $(\a_1,\xi_1,\underline{e_1})$ and $(\a_2,\xi_2,\underline{e_2})$ are equivalent if there exists a $z\in E^{\times}$ such that $\a_2=z\a_1$, $\underline{e_2}=z\underline{e_1}$ and $\xi_2=z\bar{z}\xi_1$, i.e. $[\a,\xi,\underline{e}]=[z\a,z\bar{z}\xi,z\underline{e}]$ for any $z\in E^{\times}$. Given such a pair, one can write 
\begin{equation}
\a=\O_F\alpha+\partial_F^{-1}\beta,~\Sigma(\beta/\alpha)\in\H^2,
\end{equation}
with $\xi(\bar{\alpha}\beta-\alpha\bar{\beta})=1$. 

\begin{lem}
The CM point $z_0^+$ or $z_0^-$ in Section \ref{cha:SV} associated to $\tildeE_{\tsigma_1}$ correspond to $z=\Sigma(\beta/\alpha)\in\H^2$ associated to the CM point $[\a,\xi,\underline{e}]\in\CM_2^{\Sigma}(E)$.
\end{lem}

\begin{proof}
For simplicity, let us denote $\beta/\alpha$ by $\omega$. Now for the point $z=\Sigma(\omega)=(\sigma_1(\omega),\sigma_2(\omega))\in\H^2$, recall the embedding $\phi:\H^2\to\H_2$ in Section \ref{sec:relation}, it is not hard to obtain that
$$\phi(z)=\matrix{\tau_1}{\tau_{12}}{\tau_{12}}{\tau_2}$$
with $\det(\tau)=\sigma_1(\omega)\sigma_2(\omega)\Big(\sigma_1(e_1)\sigma_2(e_2)-\sigma_1(e_2)\sigma_2(e_1)\Big)^2$,
where
\begin{eqnarray*}
\tau_1&=&\sigma_1(e_1^2\omega)+\sigma_2(e_1^2\omega),\\
\tau_2&=&\sigma_1(e_2^2\omega)+\sigma_2(e_2^2\omega),\\
\tau_{12}&=&\sigma_1(e_1e_2\omega)+\sigma_2(e_1e_2\omega).
\end{eqnarray*}

Recall we also identify $\H_2^{\pm}$ with $\D$ and $\mathcal{D}$ in Section \ref{sec:identification}. Now $\C\Xi(\phi(z))$ will be a negative line in $V(\C)$, hence corresponds to a point in $\D$. Finally, through the $\Q$-isogeny $\kappa_{\alpha,\beta}$ from $W_0$ to $\tildeE$, we get that
\begin{eqnarray*}
\kappa_{\alpha,\beta}(\Xi(\phi(z))&=&-\det(\tau)\sigma_1(\alpha)\sigma_2(\alpha)+\sigma_1(\alpha)\sigma_2(\beta)\Big(\sigma_1(e_2)\sigma_2(e_1)-\sigma_1(e_1)\sigma_2(e_2)\Big)\sigma_1(\omega)\\
&&+\sigma_1(\beta)\sigma_2(\alpha)\Big(\sigma_1(e_1)\sigma_2(e_2)-\sigma_1(e_2)\sigma_2(e_1)\Big)\sigma_2(\omega)+\sigma_1(\beta)\sigma_2(\beta)\\
&=&(1-D)\sigma_1(\beta)\sigma_2(\beta)\in\tildeE
\end{eqnarray*}
Now $\C\Xi(\phi(z))$ becomes $\tildeE\otimes_{\tildeE,\tsigma_j}\C =\C$, where $\tsigma_j$, $j=1,2,3,4$ are 4 complex embeddings of $\tildeE$. To make it a negative line, there are exactly two embeddings giving it. Recall our definition of $W$ in Section \ref{cha:SV}, the embeddings have to be $\tsigma_1$ and $\bar{\tsigma}_1$
\end{proof}

\subsubsection{Construction of \texorpdfstring{$T$}{Lg}}
Now since we have identified $W$ with $\tildeE$, it is obvious that 
$$\Res_{\tildeF/\Q}\SO(W)(\Q)=\SO(W)(\tildeF)=\tildeE^1.$$ In order for the following diagram of short exact sequences hold true, 
$$\xymatrix{
1\ar[r]	&\G_m\ar[r]\ar[d]_{=}	&\tildeT\ar[r]\ar@{^{(}->}[d]	&\Res_{\tildeF/\Q}\SO(W)\ar[r]\ar[d]	&1\\
1\ar[r]	&\G_m\ar[r]\ar[d]_{=}	&\Gamma=\GSpin(\Res_{\tildeF/\Q}W)\ar[r]\ar@{^{(}->}[d]_{\phi}	&\SO(\Res_{\tildeF/\Q}W)\ar[r]\ar[d]	&1\\
1\ar[r]	&\G_m\ar[r]	&\GSp_4(\Q)=\GSpin(V)\ar[r]	&\SO(V)\ar[r]	&1
}$$
where $\Gamma=\{g\in\GL_2(F)\mid\det g\in\Q^{\times}\}$, we must have 
\begin{eqnarray*}
\tildeT(\Q)=\{z\in E^{\times}\mid z\bar{z}\in\Q^{\times}\}&\to&\tildeE^1\\
z&\mapsto&\frac{z}{\sigma(\bar{z})}
\end{eqnarray*}
Similarly,
$$\xymatrix{
1\ar[r]	&\{\pm 1\}\ar[r]\ar[d]_{=}	&\tildeT^1\ar[r]\ar@{^{(}->}[d]	&\Res_{\tildeF/\Q}\SO(W)\ar[r]\ar[d]	&1\\
1\ar[r]	&\{\pm 1\}\ar[r]\ar[d]_{=}	&\SL_2(F)\ar[r]\ar@{^{(}->}[d]_{\phi}	&\SO(\Res_{\tildeF/\Q}W)\ar[r]\ar[d]	&1\\
1\ar[r]	&\{\pm 1\}\ar[r]	&\Sp_4(\Q)\ar[r]	&\SO(V)\ar[r]	&1
}$$
where $\tildeT^1=\tildeT\cap E^1$.

Let us summarize our construction of CM points in the following proposition.
\begin{prop}
\label{prop:summary}
Let $\boldA=[\a,\xi,\underline{e}]\in\CM_2^{\Sigma}(E)$ be as above and write $\a=\O_F\alpha+\partial_F^{-1}\beta$. We can interpret $z_0$ by $\Sigma(\beta/\alpha)\in\H^2$ and we can define $T=\{z\in E^{\times}\mid z\bar{z}\in\Q^{\times}\}$. By the construction in (\ref{eqn:cycle}), we have a CM cycle
$$Z(\boldA)=T(\Q)\bs\{z_0^{\pm}\}\times T(\A_f)/K_T,$$
which also has a $C(T)=T(\Q)\bs T(\A_f)/K_T$-action.
\end{prop}

\subsection{Siegel Theta Constants}
\label{sec:theta}
Let $z\in\H_2$ and a quadruple $(x_1,x_2,y_1,y_2)\in\Z^4$, which we write as $(\va,\vb)=((x_1,x_2),(y_1,y_2))$, i.e. $\va,\vb\in\Z^2$, the Siegel theta constant of characteristic $(\va,\vb)$ is defined as
$$\thetaS_{\va,\vb}(z)=\sum_{m\in\Z^2}\exp\left(\pi i\left(m+\frac{\va}{2}\right)z\left(m+\frac{\va}{2}\right)^t+2\pi i\left(m+\frac{\va}{2}\right)\left(\frac{\vb}{2}\right)^t\right).$$

The function $\thetaS_{\va,\vb}(z)$ is a Siegel modular form of weight $\dfrac{1}{2}$ for the modular group $\Gamma(2)$. Note that
$$\thetaS_{\va,\vb}(z)=\pm\thetaS_{\va',\vb'}(z)\text{ if }(\va,\vb)\equiv (\va',\vb')~(\mod~2).$$
From now on, we only focus on the case where $\va,\vb\in\{0,1\}^2$. The sign ambiguity will not be an issue, since only $(\thetaS_{\va,\vb})^2$ appears in our future computations. A quadruple $(\va,\vb)$ as above is called even if $\va^t\vb=0$, i.e. $x_1y_1+x_2y_2\equiv 0~(\mod~2)$. There are 10 even quadruples and it is well-known that $\thetaS_{\va,\vb}\not=0$ if and only if $(\va,\vb)$ is even.

Recall that for our construction of the lattice $L$ in $V$ in Section \ref{sec:realization}, $L'/L\cong\left(\dfrac{1}{2}\Z\Big/\Z\right)^4\oplus\left(\dfrac{1}{4}\Z\Big/\Z\right)$ has a basis of 64 elements $\varphi_1,\dots,\varphi_{64}$. 

For this choice of $L$ and each even pair $(\va,\vb)$, Lippolt constructed in \cite{Lip08} an weakly holomorphic modular form $f_{\va, \vb}$ of weight $-1/2$ valued in $S_L$ such that 
$$\theta_{\va, \vb}^S(z)=\Psi(z, f_{\va, \vb})\text{ or }-\log\|\thetaS_{\va, \vb}(z)\|_{\Pet}^2=\Phi(z,f_{\va, \vb})$$
is the Borcherds lifting of $f_{\va,\vb}$ in the fashion of Theorem \ref{thm:Borcherds}. Here
$f_{\va,\vb}=\sum_{\mu\in L'/L}f_{\va,\vb,\mu}\varphi_{\mu},$
where $f_{\va,\vb,\mu}\in\{0,\pm\mathfrak{u},\pm\mathfrak{v},\mathfrak{w}\}$ and
\begin{eqnarray*}
\u&=&1+4q+14q^2+40q^3+100q^4+232q^5+504q^6+\cdots\\
\v&=&-2q^{\frac{1}{2}}-8q^{\frac{3}{2}}-24q^{\frac{5}{2}}-64q^{\frac{7}{2}}-154q^{\frac{9}{2}}-344q^{\frac{11}{2}}-\cdots\\
\w&=&q^{-\frac{1}{8}}-q^{\frac{7}{8}}+q^{\frac{15}{8}}-2q^{\frac{23}{8}}+3q^{\frac{31}{8}}-4q^{\frac{39}{8}}+5q^{\frac{47}{8}}-7q^{\frac{55}{8}}+\cdots
\end{eqnarray*}
where $q=e^{2\pi i\tau}$.

Please check Appendix \ref{app:Lippolt} for the exact definitions of $~\mathfrak{u},~\mathfrak{v},~\mathfrak{w},f_{\va,\vb}$.

With our realization of $V$ and $W=\tildeE$, we have constructed $z_0$ and $T$ in the previous sections. Now by applying our main Theorem \ref{thm:lattice} to this case, we will have the following
\begin{prop}
\label{prop:theta}
Notation as above and $Z(\boldA)$ as defined in Proposition \ref{prop:summary}, we have
$$-\log\|\theta_{\va,\vb}^S(Z(\boldA))\|_{\Pet}^2=C_E\cdot\CT\left[\sum_{\mu\in L'/L}\sum_{(\mu_0,\mu_1)\in\varpi^{-1}(\mu)}f_{\va,\vb,\mu}\theta_0(\varphi_{\mu_0,L_0})\E_W(\varphi_{\mu_1,M})\right],$$
where 
$$C_E=\frac{\deg(Z(\boldA))}{\Lambda(0,\chi)}=\frac{4}{\omega_E}\frac{|C(T)|}{\Lambda(0,\chi)},$$
with $\omega_E$ being the number of roots of unity in $E$.
\end{prop}

\begin{rmk}
By the definition in Section \ref{cha:SV}, every point in $Z(\boldA)$ is counted with multiplicity $\dfrac{2}{\omega_E}$. Furthermore, $z_0^{\pm}$ are the same point in $X_2(2)$, the image of a point in $X_2(2)$ should be counted with multiplicity $\dfrac{4}{\omega_E}$.
\end{rmk}

\begin{rmk}
Note that $Z(\boldA)$ is a torus orbit that always contains a Galois orbit but might be strictly smaller than a whole CM cycle $\CM_2(E)$.
\end{rmk}

\subsection{Rosenhain Invariants}
Consider a genus 2 curve with the form $C:y^2=\prod_{i=1}^6(x-u_i)$ over the algebraic closure, any three of the $u_i$ can be mapped to 0,1 and $\infty$ using linear fractional transformations to write $C$ in Rosenhain form as
\begin{equation}
C_{\lambda}:y^2=x(x-1)(x-\lambda_1)(x-\lambda_2)(x-\lambda_3),
\end{equation}
where $\lambda_1,~\lambda_2$ and $\lambda_3$ are called the Rosenhain invariants of $C$. 

Let $\M_2$ be the moduli space of genus two curves. Similarly to the modular curve case, there is a map $\M\to X_2(1)=\Sp_4(\Z)\bs\H_2$ that sends $C\in\M$ to its Jacobian $J(C)$ with proper structures. Now let $\M_2(2)=\{C\in\M_2|C\to\P^1\text{ branch points are rational}\}$. We have the following diagram
$$\xymatrix{
\M_2\ar[d]\ar[r]&X_2(1)\ar[d]\\
\M_2(2)\ar[r]&X_2(2)	
}$$

In other words, for some special point $\tau\in X_2(2)$, we should also be able to construct an associated genus 2 curve $C_{\tau}$ by computing all of its Rosenhain invariants $\lambda_1(\tau),\lambda_2(\tau),\lambda_3(\tau)$. 

For generating such curves, there are many different possible combinations of 6 even theta constants which yield a Rosenhain model for the curve $C_{\lambda}$ with
\begin{equation}
\label{eqn:lambda}
\lambda_1=-\frac{\theta_{i_1}^2\theta_{i_3}^2}{\theta_{i_4}^2\theta_{i_6}^2},~
\lambda_2=-\frac{\theta_{i_2}^2\theta_{i_3}^2}{\theta_{i_5}^2\theta_{i_6}^2},~
\lambda_3=-\frac{\theta_{i_1}^2\theta_{i_2}^2}{\theta_{i_4}^2\theta_{i_5}^2}.
\end{equation}

Several different choices for the combination of even theta constants have appeared in the literature. In this work we find it convenient to adopt the following choices of of $(i_1,\dots,i_6)$ as our choices.
$$\begin{array}{lll}
i_1=(0,0,1,0),&i_2=(1,0,0,0),&i_3=(0,1,1,0),\\
i_4=(0,0,1,1),&i_5=(1,1,0,0),&i_6=(1,0,0,1).
\end{array}$$

Hence, 
\begin{cor}
\label{cor:Rosenhain}
Notation as above. We have
$$\log|\lambda_i(Z(\boldA))|=C_E\cdot\CT\left[\sum_{\mu\in L'/L}\sum_{(\mu_0,\mu_1)\in\varpi^{-1}(\mu)}f_{i,\mu}\theta_0(\varphi_{\mu_0,L_0})\E_W(\varphi_{\mu_1,M})\right],
$$
where 
\begin{eqnarray*}
f_1&=&-f_{0,0,1,0}-f_{0,1,1,0}+f_{0,0,1,1}+f_{1,0,0,1},\\
f_2&=&-f_{1,0,0,0}-f_{0,1,1,0}+f_{1,1,0,0}+f_{1,0,0,1},\\
f_3&=&-f_{0,0,1,0}-f_{1,0,0,0}+f_{0,0,1,1}+f_{1,1,0,0}.
\end{eqnarray*}
\end{cor}

\subsection{Hilbert Siegel Constants}
With notations as above, for $(\va,\vb)\in (\O_F/2\O_F)^2$ and $z\in\H^2$, we define the Hilbert theta constant of characteristic $(\va,\vb)$ as
\begin{equation}
\thetaH_{\va,\vb}(z)=\sum_{u\in\O_F}\exp\left(\pi i\tr_{F/\Q}\left(\left(u+\frac{\va}{2}\right)^2z+\left(u+\frac{\va}{2}\right)\frac{\vb}{\sqrt{D}}\right)\right)
\end{equation}

Similarly to the Siegel theta constants, the function $\thetaH_{\va,\vb}(z)$ is a Hilbert modular form of weight $\dfrac{1}{2}$ and
$$\thetaH_{\va,\vb}=\pm\thetaH_{\va',\vb'}\text{ if }(\va,\vb)\equiv (\va',\vb')~(\mod~2\O_F).$$ 
One can prove that $\thetaH_{\va,\vb}\not=0$ if and only if $\tr_{F/\Q}\left(\dfrac{\va\vb}{\sqrt{D}}\right)\in 2\Z$. In this case, we call the pair $(\va,\vb)$ even.

Now it can be seen that Lemma \ref{lem:relation} actually shows that the pull back via $\phi$ of a Siegel modular form $\thetaS$ for a congruence subgroup of $\Sp_4(\Z)$ is a Hilbert modular form $\thetaH$ for a congruence subgroup of $\SL_2(\O_F\oplus\partial_F^{-1})$.

\begin{thm}[\cite{LNY16}*{Theorem 2.2}]
\label{thm:identity}
Let the notation be as above. Then
$$\phi^*\thetaS_{\va,\vb}=\pm\thetaH _{\phi^{-1}(\va,\vb)}.$$
\end{thm} 

Given that Lippolt in \cite{Lip08} has found the modular forms whose Borcherds lifts give out Siegel theta constants $\thetaS$. It is not hard to find the corresponding modular forms associated with Hilbert theta constants $\thetaH$.

Associated to the quadratic space $(V,Q_V)$ is the seesaw dual pair $(\O(V),\SL_2)$.

$$\xymatrix{
\widetilde{\SL_2}\ar[dr]\ar@{^{(}->}[d]&\O(\Res_{\tildeF/\Q}W)\times\O(V_0)\ar[dl]\ar@{^{(}->}[d]\\
\SL_2\times\widetilde{\SL_2}\ar[ur]&\O(V)\ar[ul]
}$$

Recall that the Borcherds lifts as regularized theta lifts as follows
\begin{eqnarray*}
\thetaS_{\va,\vb}=\Phi(z,g;f^S_{\va,\vb})
&=&\int^{\reg}_{\F}\langle f^S_{\va,\vb}(\tau),\theta_V(\tau,z,g)\rangle\md\mu(\tau),\\
&=&\int^{\reg}_{\F}\langle f^S_{\va,\vb}(\tau),\theta_{\Res_{\tildeF/\Q}W}(\tau,z,g)\otimes\theta_{V_0}(\tau)\rangle\md\mu(\tau)\\
&=&\int^{\reg}_{\F}\langle f^S_{\va,\vb}(\tau)\theta_{V_0}(\tau),\theta_{\Res_{\tildeF/\Q}W}(\tau,z,g)\rangle\md\mu(\tau),\\
\thetaH_{\va,\vb}=\Phi(z,g;f_{\va,\vb}^H)
&=&\int^{\reg}_{\F}\langle f^H_{\va,\vb}(\tau),\theta_{\Res_{\tildeF/\Q}W}(\tau,z,g)\rangle\md\mu(\tau).
\end{eqnarray*}
So we can take $f^H_{\va,\vb}=f^S_{\va,\vb}\theta_{V_0}$.

\section{Main Formulas for Siegel 3-Fold Case}
Now we would like to compute the terms in the formulas appearing in Proposition \ref{prop:theta} and Corollary \ref{cor:Rosenhain}.

\begin{lem}
\label{lem:case}
Assume $(\va,\vb)\in\{0,1\}^4$ and $\mu_0\in L'_0/L_0,~\mu_1\in M'/M,~\mu\in L'/L$ satisfying $(\mu_0,\mu_1)\in\varpi^{-1}(\mu)$ in $(L'_0\oplus M')/(L_0\oplus M)$ and 
$$\CT[f_{\va,\vb,\mu}\theta_0(\varphi_{\mu_0,L_0})\E_W(\varphi_{\mu_1,M})]\not=0.$$ Then only the following cases may happen:
\begin{enumerate}
	\item[(0)] $\mu=(a,b,0,0,r)$ with $\mu_0=(0)$ and $\mu_1=(a,b,r,0)$. Furthermore, $\theta_0(\varphi_{\mu_0})$ has constant term with the corresponding terms in $\E_W(\varphi_{\mu_1})$ and $f_{\va,\vb,\mu}$ being their constant terms.
	\item[(1)] $\mu=(a,b,0,0,r)$ with $\mu_0=(0)$ and $\mu_1=(a,b,r,0)$. Furthermore, $\theta_0(\varphi_{\mu_0})$ has constant term with the corresponding term in $\E_W(\varphi_{\mu_1})$ being $q^{1/8}$, the corresponding term in $f_{\va,\vb,\mu}$ being $q^{-1/8}$;
	\item[(2)] $\mu=\left(a,b,\dfrac{1}{2},0,r\right)$ with $\mu_0=\left(-\dfrac{1}{4D}\right)$ and $\mu_1=\left(a,b,r,\dfrac{1}{4D}\right)$. Furthermore, $\theta_0(\varphi_{\mu_0})$ has leading term $q^{1/8D}$ with the degree of corresponding term in $\E_W(\varphi_{\mu_1})$ being $q^{1/8-1/8D}$, the corresponding term in $f_{\va,\vb,\mu}$ being $q^{-1/8}$.
\end{enumerate}
\end{lem}

\begin{proof}
First, let us write $\mu=(a,b,c,d,r)\in L'/L$, $\mu_1=(a,b,r,s)\in M'/M$ and $\mu_0=(t)\in L_0'/L_0$.

Note that $V_0$ and $W_0$ are orthogonal to each other under $Q_V$ as subspaces of $V$. In order that $(\mu_0,\mu_1)\in\varpi^{-1}(\mu)$, one requires that $\mu_0+\mu_1=\mu\in(L_0'\oplus M')/(L_0\oplus M)$, or more concretely,
$$c=D(s-t),\quad d=s+t.$$
In other words,
$$s=\frac{c+Dd}{2D},\quad t=\frac{-c+Dd}{2D}.$$

Then simple calculation shows that there are only four possible cases for the pair ($\mu_1,\mu_0)$ as follows.

\begin{enumerate}
	\item[(i)] $c=0,~d=0$: we will have $s=0,~t=0$.
	\item[(ii)] $c=0,~d=\dfrac{1}{2}$: we will have $s=\dfrac{1}{4},~t=\dfrac{1}{4}$.
	\item[(iii)] $c=\dfrac{1}{2},~d=0$: we will have $s=\dfrac{1}{4D},~t=-\dfrac{1}{4D}$.
	\item[(iv)] $c=\dfrac{1}{2},~d=\dfrac{1}{2}$: we will have $s=\dfrac{1}{4}+\dfrac{1}{4D},~t=\dfrac{1}{4}-\dfrac{1}{4D}$.
\end{enumerate}
By Table \ref{tab:first}-\ref{tab:last}, the only negative term in $f_{\va,\vb,\mu}$ is $q^{-1/8}$. In case (ii) and (iv), the degrees of the leading term of $\theta_0(\varphi_{\mu_0})$ are 
$$2D\left(\frac{1}{4}\right)^2=\frac{D}{8}\text{ and }2D\left(\frac{1}{4}-\frac{1}{4D}\right)^2=\frac{(D-1)^2}{8D},$$
which are greater than $1/8$ as long as $D\geq 3$. Therefore, only when $\mu$ has $d$-coordinate 0, we can find two pairs $(\mu_1,\mu_0)$ such that $\mu_1+\mu_0=\mu$. Otherwise, there is no such pair for $\mu$ and hence the corresponding term makes no contribution to the sum.

For the other two cases, it is also very easy to compute that the degrees of the leading term of $\theta_0(\varphi_{\mu_0})$ are 
$$0\text{ and }2D\left(\frac{1}{4D}\right)^2=\frac{1}{8D}.$$
Recall that the degree of the negative term in $f_{\va,\vb,\mu}$ is -1/8. In order to get the constant term, the degrees of the corresponding term in Eisenstein series in case (i) and (iii) should be
$$\frac{1}{8}\text{ and }\frac{1}{8}-\frac{1}{8D}.$$

Case (0) is obvious with similar argument.
\end{proof}

Now we would like to state our main formulas for special values of Siegel theta constants and Rosenhain invariants.

Let $(\va,\vb)\in\{0,1\}^4$ be an even pair. By Table \ref{tab:first}-\ref{tab:last} in Appendix \ref{app:Lippolt}, there are exactly two $\mu$ such that $f_{\va,\vb,\mu}=\w$ has negative degree term $q^{-1/8}$. Let us denote them by $\mu(\va,\vb,1),\mu(\va,\vb,2)$.

Recall in Proposition \ref{prop:theta}, we have $\theta^S_{\va,\vb}=\Psi(f_{\va,\vb})$, with
\begin{eqnarray*}
f_{\va,\vb}&=&\sum_{i=1}^{64}f_{\va,\vb,\mu(i)}\varphi_{\mu(i)}\\
&=&\sum_{i=1}^2\varphi_{\mu(\va,\vb,i)}q^{-1/8}+\sum_{i=1}^{16}\varphi_{\mu(i)}+\text{higher terms},
\end{eqnarray*}
where $\mu(i):\{1,\dots,64\}\to L'/L$ is given by Table \ref{tab:label} in Appendix \ref{app:Lippolt}.

\begin{thm}
\label{thm:MainF}
Let the notation be as above. For $\boldA=[\a,\xi,\underline{e}]\in\CM_2^{\Sigma}(E)$ and $Z(\boldA)$ be the associated CM cycle defined in Proposition \ref{prop:summary}. Now we have
\begin{eqnarray*}
&&-\log||\thetaS_{\va,\vb}(Z(\boldA))||_{\Pet}^2=C_E\sum_{i=1}^4\epsilon_{\va,\vb,4i-3}a_0(\varphi_{\mu_1(4i-3)})\\
&+&C_E\sum_{i=1}^2\left(
\delta(\va,\vb,i)\sum_{\substack{t\in F_+^{\times}, \\ \tr_{F/\Q}t=\frac{1}{8}}}a\left(t,\varphi_{\mu_1(\va,\vb,i)}\right)
+2\delta'(\va,\vb,i)\sum_{\substack{t\in F_+^{\times}, \\ \tr_{F/\Q}t=\frac{1}{8}-\frac{1}{8D}}}a\left(t,\varphi_{\mu_1(\va,\vb,i)}\right)\right)
\end{eqnarray*}
where $C_E=\dfrac{4}{\omega_E}\dfrac{|C(T)|}{\Lambda(0,\chi)}$,
with $\omega_E$ being the number of roots of unity in $E$ and $C(T)=T(\Q)\bs T(\A_f)/K_T$ and $\epsilon_{\va,\vb,i}=f_{\va,\vb,\mu(i)}/\u\in\{\pm 1\}$ and
\begin{eqnarray*}
\delta(\va,\vb,i)&=&\left\{\begin{array}{ll}
1&\text{if }\mu(\va,\vb,i)=(a,b,0,0,r),\\
0&\text{otherwise},
\end{array}\right.\\
\delta'(\va,\vb,i)&=&\left\{\begin{array}{ll}
1&\text{if }\mu(\va,\vb,i)=\left(a,b,\dfrac{1}{2},0,r\right),\\
0&\text{otherwise},
\end{array}\right.\\
\mu_1(i)&=&(a,b,r,0)\quad\text{if }\mu(i)=(a,b,0,0,r),\\
\mu_1(\va,\vb,i)&=&\left\{\begin{array}{ll}
(a,b,r,0)&\text{if }\mu(\va,\vb,i)=(a,b,0,0,r),\\
\left(a,b,r,\dfrac{1}{4D}\right)&\text{if }\mu(\va,\vb,i)=\left(a,b,\dfrac{1}{2},0,r\right).
\end{array}\right.
\end{eqnarray*} 
\end{thm}

\begin{proof}
By Proposition \ref{prop:theta}, we get that
$$-\log\|\theta_{\va,\vb}^S(Z(\boldA))\|_{\Pet}^2=C_E\cdot\CT\left[\sum_{\mu\in L'/L}\sum_{(\mu_0,\mu_1)\in\varpi^{-1}(\mu)}
f_{\va,\vb,\mu}\theta_0(\varphi_{\mu_0,L_0})\E_W(\varphi_{\mu_1,M})\right].$$

By Table \ref{tab:first}-\ref{tab:last} in Appendix \ref{app:Lippolt}, we know that for each $(\va,\vb)\in\{0,1\}^4$, there are exactly two $\mu$ so that $f_{\va,\vb,\mu}=\w$ has negative degree term $q^{-1/8}$ with coefficient 1. By lemma \ref{lem:case}, there are only three cases we need to care about.

Recall the Fourier expansion of $\E_W(\varphi_{\mu_1})$ in Proposition \ref{prop:approx},
$$\E_W(\tau,\varphi_{\mu_1})=a_0(\varphi_{\mu_1})+\sum_{m\in\Q_+}a_m(\varphi_{\mu_1})q^m,$$
where
$$a_m(\varphi)=\sum_{t\in F_+^{\times},\tr_{F/Q}t=m}a(t,\varphi).$$

By definition of Siegel theta functions in Section \ref{cha:theta}, for the special case of dimension 1 in our case of $V_0$, we have
$$\theta_0(\tau,\varphi_{\mu_0,L_0})=\sum_{m\in\mu_0+L}b_m(\varphi_{\mu_0,L_0})q^m.$$
It is not hard to show that $b_m(\varphi_{\mu_0})=2$ if $m\not=0$ and $b_0(\varphi_{\mu_0})=1$ if $\mu_0=0$ or $b_0(\varphi_{\mu_0})=0$ if $\mu_0\not=0$. By Table \ref{tab:label} in Appendix \ref{app:Lippolt}, only the first 16 coordinates of $f_{\va,\vb}$ has constant terms 1 or -1. Therefore,
\begin{eqnarray*}
&&-\log\|\theta_{\va,\vb}^S(Z(\boldA))\|_{\Pet}^2\\
&=&C_E\sum_{i=1}^{16}\sum_{\mu_0+\mu_1=\mu_i}\epsilon_{\va,\vb,i}b_0(\varphi_{\mu_0})a_0(\varphi_{\mu_1})
+C_E\sum_{i=1}^2\sum_{\mu_0+\mu_1=\mu_{\va,\vb,i}}
\left(b_0(\varphi_{\mu_0}))a_{\frac{1}{8}}+b_{\frac{1}{8D}(\varphi_{\mu_0})a_{\frac{1}{8}-\frac{1}{8D}}}(\varphi_{\mu_1})\right)
\end{eqnarray*}

For case (0) in Lemma \ref{lem:case}, we have to require $\mu(i)=(a,b,0,0,r)$, then we get $\mu_0=(0)$ and $\mu_1=(a,b,r,0)$. By looking up Table \ref{tab:label}, among the first 16 coordinates, only 1st, 5th, 9th, 13th qualify for the condition.

For case (1), we have to require $\mu_{\va,\vb,i}=(a,b,0,0,r)$, then we get $\mu_0=0$, $\mu_1=(a,b,r,0)$. In this case, only the first term in the parentheses contribute to the sum and $b_0(\varphi_{\mu_0})=1$.

For case (2), we have to require $\mu_{\va,\vb,i}=(a,b,1,0,r)$, then we get $\mu_0=-\dfrac{1}{4D}$, $\mu_1=\left(a,b,r,\dfrac{1}{4D}\right)$. In this case, only the second term in the parentheses contribute to the sum and $b_{\frac{1}{8D}}(\varphi_{\mu_0})=2$.

For all other cases, by Lemma \ref{lem:case}, the contribution of the corresponding term will be 0.

Therefore, by putting all of above data in, we obtain our main formula
\begin{eqnarray*}
&&-\log||\thetaS_{\va,\vb}(Z(\boldA))||_{\Pet}^2=C_E\sum_{i=1}^4\epsilon_{\va,\vb,4i-3}a_0(\varphi_{\mu_1(4i-3)})\\
&+&C_E\sum_{i=1}^2\left(
\delta(\va,\vb,i)\sum_{\substack{t\in F_+^{\times}, \\ \tr_{F/\Q}t=\frac{1}{8}}}a\left(t,\varphi_{\mu_1(\va,\vb,i)}\right)
+2\delta'(\va,\vb,i)\sum_{\substack{t\in F_+^{\times}, \\ \tr_{F/\Q}t=\frac{1}{8}-\frac{1}{8D}}}a\left(t,\varphi_{\mu_1(\va,\vb,i)}\right)\right)
\end{eqnarray*}
\end{proof}

\begin{thm}
\label{thm:MainF2}
Let the notation be as above. For $k=1,2,3$, $\lambda_k$ is as defined in (\ref{eqn:lambda})
$$\lambda_1=-\frac{\theta_{i_1}^2\theta_{i_3}^2}{\theta_{i_4}^2\theta_{i_6}^2},~
\lambda_2=-\frac{\theta_{i_2}^2\theta_{i_3}^2}{\theta_{i_5}^2\theta_{i_6}^2},~
\lambda_3=-\frac{\theta_{i_1}^2\theta_{i_2}^2}{\theta_{i_4}^2\theta_{i_5}^2},$$
with our special choice of $(i_1,\dots,i_6)$ as
$$\begin{array}{lll}
i_1=(0,0,1,0),&i_2=(1,0,0,0),&i_3=(0,1,1,0),\\
i_4=(0,0,1,1),&i_5=(1,1,0,0),&i_6=(1,0,0,1).
\end{array}$$
Then we have
\begin{eqnarray*}
&&\log|\lambda_k(Z(\boldA))|\\
&=&C_E\sum_{(\va,\vb)\in S_k}\varepsilon_{k,\va,\vb}\sum_{i=1}^2\left(
\delta(\va,\vb,i)\sum_{\substack{t\in F_+^{\times}, \\ \tr_{F/\Q}t=\frac{1}{8}}}a\left(t,\varphi_{\mu_1(\va,\vb,i)}\right)
+2\delta'(\va,\vb,i)\sum_{\substack{t\in F_+^{\times}, \\ \tr_{F/\Q}t=\frac{1}{8}-\frac{1}{8D}}}a\left(t,\varphi_{\mu_1(\va,\vb,i)}\right)\right)
\end{eqnarray*}
where
$$S_1=\{i_1,i_3,i_4,i_6\},\quad S_2=\{i_2,i_3,i_5,i_6\},\quad S_3=\{i_1,i_2,i_4,i_5\}$$
and 
\begin{center}
\begin{tabular}{c|c|c|c|c|c|c}
$(\va,\vb)$&$i_1$&$i_2$&$i_3$&$i_4$&$i_5$&$i_6$\\ \hline
$\varepsilon_{1,\va,\vb}$&-1&0&-1&1&0&1\\ \hline
$\varepsilon_{2,\va,\vb}$&0&-1&-1&0&1&1\\ \hline
$\varepsilon_{3,\va,\vb}$&-1&-1&0&1&1&0
\end{tabular}
\end{center}

\end{thm}

\begin{rmk}
For our choice of $(i_1,\dots,i_6)$, it turns out that the terms $a_0(\varphi_{\mu_1(4i-3)})$ in the previous theorem all cancel out thanks to the definition of Rosenhain invariant. For other choices, this might not happen.
\end{rmk}

\subsection{Local Whittaker Functions}
To compute the $t$-th Fourier coefficient $a(t,\varphi)$ of $E^{*,\prime}(\vectau,0,\varphi,\vecid)$, one may assume that $\phi=\otimes\phi_v$ is factorizable by linearity. Write for $t\not=0$, 
$$E^*(\vectau,0,\varphi,\vecid)=\prod_{v\nmid\infty}W_{t,v}^*(s,\varphi)\prod_{j=1}^2W_{t,\sigma_j}(\tau_j,s,\Phi_{\sigma_j}),$$
where 
$$W_{t,v}^*(s,\varphi)=|A|_v^{-\frac{s+1}{2}}L_v(s+1,\chi_v)W_{t,v}(s,\varphi)$$
with 
$$W_{t,v}(s,\varphi)=\int_{F_v}\omega(w^{-1}n(b))(\varphi_v)(0)|a(w^{-1}n(b))|^s_v\psi_v(-tb)\md b,$$
and
$
W_{t,\sigma_j}^*(\tau_j,s,\Phi_{\sigma_j})=v_j^{-\frac{1}{2}}\pi^{-\frac{s+1}{2}}\Gamma\left(\frac{s+2}{2}\right)\int_{\R}\Phi_{\sigma_j}(wn(b)g_{\tau_j},s)\psi(-bt)\md b.$
Here $A=\mathrm{N}_{F/\Q}(\partial_F d_{E/F})$ and $|a(g)|_{\p}=|a|_{\p}$ if $g=n(b)m(a)k$ with $k\in\SL_2(\O_{\p})$.

The following proposition is well-known and is recorded here for reference.
\begin{prop}[\cite{YY19}*{Proposition 2.7}]
\label{prop:Diff}
For a totally positive number $t\in F$, let
$$\Diff(W,t)=\{\p\mid W_{\p}\text{ does not represent }t\}$$
be the so-called 'Diff' set of Kudla. Then $|\Diff(W,t)|$ is finite and odd. Moreover,
\begin{enumerate}
	\item If $|\Diff(W,t)|>1$, then $a(t,\varphi)=0$.
	\item If $\Diff(W,t)=\{\p\}$, then $W_{t,\p}^*(0,\varphi)=0$, and 
	$$a(t,\varphi)=-4W_{t,\p}^{*,\prime}(0,\varphi)\prod_{q\nmid p\infty}W_{t,\q}^*(0,\varphi).$$
	\item When $\p\nmid\alpha A$ and $\varphi_p=\mathrm{char}(\O_{E_{\p}})$, $W_{t,\p}^*(s,\phi)=0$ unless $t\in\O_{\p}$. In this case, one has
	$$\frac{W_{t,\p}^*(0,\varphi)}{\gamma(W_{\p})}=
	\left\{\begin{array}{ll}
	1+\ord_{\p}(t\sqrt{D})&\text{if $\p$ is split in $E$},\\
	\frac{1+(-1)^{\ord_{\p}(t\sqrt{D})}}{2}&\text{if $\p$ is inert in $E$}.
	\end{array}\right.$$
	Here $\gamma(W_{\p})$ is the local Weil index (an 8-th root of unity) associated to the Weil representation. Moreover, in this case, $W_{t,\p}^*(0,\varphi)=0$ if and only if $\ord_{\p}(t)$ is odd and $\p$ is inert in $E$. In such a case, one has
	$$\frac{W_{t,\p}^{*,\prime}(0,\varphi)}{\gamma(W_{\p})}=\frac{1+\ord_{\p}(t\sqrt{D})}{2}\log\mathrm{N}(\p).$$
	\item One has for $j=1,2$,
	$$W_{t,\sigma_j}^*(\tau,0,\Phi_{\sigma_j})=-2ie(t\tau).$$
\end{enumerate}
\end{prop}

\subsection{Computation of \texorpdfstring{$a(t,\varphi)$}{Lg}}
The calculation of $a\left(\frac{t}{\sqrt{D}},\varphi'_{\va,\vb,i}\right)$ is local and is similar to other cases which have been done by several authors,  such as Section 4.6 in \cite{HY12}, Section 5.3 in \cite{YY19} and Section 4.4 in \cite{YYY18}. For this reason, we only care about primes $\p\mid 2$. 

For $\varphi_{\va,\vb,i}=\varphi_{\mu_1}$, by Proposition \ref{prop:Diff}, $a\left(\frac{t}{\sqrt{D}},\varphi_{\mu_1}\right)=0$ unless $t-2\mu_1\bar{\mu_1}\in\O_F$. When $|\Diff(W,\frac{t}{\sqrt{D}})|>1$, $\varphi_{\va,\vb,i}=\varphi_{\mu_1}=0$. When $\Diff(W,\frac{t}{\sqrt{D}})=\{p\}$, $\p$ is inert in $E/F$ and $\ord_{\p}(t)$ is odd. 
However, Proposition \ref{prop:lattice} implies that $M$ is not $\O_F$-lattice. So $\varphi_{\mu_1}$ is not totally factorizable over primes of $F$ as assumed in Proposition \ref{prop:Diff}. Indeed, one only has
$$\varphi_{\mu_1}=\varphi_{\mu_1,2}\prod_{\p\nmid 2}\varphi_{\mu_1,\p},$$
where $\varphi_{\mu_1,\p}=\mathrm{char}(\O_{E,\p})$ for a prime ideal $\p$ of $F$ not dividing 2, and $\varphi_{\mu_1,2}=\mathrm{char}(M_2)$. So we need to take special care for the local computation at $p=2$.

\underline{Case (I)}: If 2 splits in $\tildeE$ completely. Write
$$2\O_{\tildeF}=\p_1\p_2,\quad\p_i=\B_i\bar{\B_i}.$$

Let $\sqrt{\tildeD}\in\Z_2$ and $\sqrt{\tDelta}\in\Z_2$ be some prefixed square roots of $\tildeD$ and $\tDelta$ respectively. We identify $\tildeF_{\p_i},~\tildeE_{\B_i}$ and $\tildeE_{\bar{B_i}}$ with $\Q_2$ as follows.

\begin{eqnarray*}
\tildeF_{\p_i}\cong\Q_2,&&\sqrt{\tildeD}\mapsto (-1)^{i-1}\sqrt{\tildeD},\\
\tildeE_{\B_i}\cong\Q_2,&&\sqrt{\tildeD}\mapsto (-1)^{i-1}\sqrt{\tildeD},~\sqrt{\tDelta}\mapsto\sqrt{\tDelta},\\
\tildeE_{\bar{\B_i}}\cong\Q_2,&&\sqrt{\tildeD}\mapsto (-1)^{i-1}\sqrt{\tildeD},~\sqrt{\tDelta}\mapsto -\sqrt{\tDelta}.
\end{eqnarray*}

With this identification, we can check that $M_2=M\otimes_Z\Z_2$ is given by
$$M_2=\{(x_1,x_2,x_3,x_4)\in\O_{\tildeE_{\B_1}}\times\O_{\tildeE_{\bar{\B_1}}}\times\O_{\tildeE_{\B_2}}\times\O_{\tildeE_{\bar{\B_2}}}\cong\Z_2^4\mid\sum x_i\in 2\Z_2\},$$
with quadratic form
$$Q(x)=\frac{x_1x_2}{\sqrt{\tildeD}}+\frac{x_3x_4}{-\sqrt{\tildeD}}=Q_{\p_1}(x_1,x_2)+Q_{\p_2}(x_3,x_4).$$

The embedding from $M$ to $M_2$ is given by 
$$x\mapsto (\sigma_1(x),\sigma_1(\bar{x}),\sigma_2(x),\sigma_2(\bar{x})).$$
So
$$\varphi_{\mu_1,2}=\mathrm{char}(M_2)=\varphi_{\p_1,0}\varphi_{\p_2,0}+\varphi_{\p_1,1}\varphi_{\p_2,1},$$
where $S_a=\{(x_1,x_2)\in\Z_2^2\mid x_1+x_2\equiv a~(\mod~2)\}$ and $\varphi_{\p_i,a}=\mathrm{char}(S_a)\in S(\Q_2^2)\cong S(\tildeE_{\p_i}).$

Correspondingly, we have
$$\varphi_{\mu_1}=\varphi_{\mu_1,0}+\varphi_{\mu_1,1},\quad a(t,\varphi_{\mu_1})=a(t,\varphi_{\mu_1,0})+a(t,\varphi_{\mu_1,1}),$$
where $\varphi_{\mu_1,i}=\varphi_{\p_1,i}\varphi_{\p_2,i}\prod_{\p\nmid 2}\varphi_{\p}$.

Proposition \ref{prop:Diff} implies
$$a(t,\varphi_{\mu_1,i})=-4\sum_{\p\text{ inert in }\tildeE/\tildeF}\frac{1+\ord_{\p}(t\sqrt{\tildeD})}{2}\prod_{\q\nmid 2}\rho_{\q}(t\p^{-1}\partial_{\tildeF})\prod_{j=1}^2\frac{W_{t\sqrt{\tildeD},\p_j}^{*,\psi'}(0,\varphi_{\p_j,i})}{\gamma(W_{\p_j})}\log(\mathrm{N}(\p)).$$

\underline{Case (II)}: Write
$$2\O_{\tildeF}=\p_1\p_2,\quad\p_1=\B_1\bar{\B_1}\text{ splits in }\tildeE,\quad\p_2\text{ inert in }\tildeE.$$

Let $\sqrt{\tildeD}\in\Z_2$ and $\sqrt{\tDelta}$ be some prefixed square roots of $\tildeD$ and $\tDelta$ respectively. We identify $\tildeF_{\p_i},~\tildeE_{\B_1},~\tildeE_{\bar{B_1}}$ with $\Q_2$ and $\tildeE_{\p_2}$ with $\Q_2(\sqrt{\tDelta})$ as follows.

\begin{eqnarray*}
\tildeF_{\p_i}\cong\Q_2,&&\sqrt{\tildeD}\mapsto (-1)^{i-1}\sqrt{\tildeD},\\
\tildeE_{\B_1}\cong\Q_2,&&\sqrt{\tildeD}\mapsto\sqrt{\tildeD},~\sqrt{\tDelta}\mapsto\sqrt{\tDelta},\\
\tildeE_{\bar{\B_1}}\cong\Q_2,&&\sqrt{\tildeD}\mapsto\sqrt{\tildeD},~\sqrt{\tDelta}\mapsto -\sqrt{\tDelta},\\
\tildeE_{\p_2}\cong\Q_2(\sqrt{\tDelta}),&&\sqrt{\tildeD}\mapsto -\sqrt{\tildeD},~\sqrt{\tDelta}\mapsto\sqrt{\tDelta}.
\end{eqnarray*}

With this identification, we can check that $M_2$ is given by
$$M_2=\{(x_1,x_2,x_3)\in\O_{\tildeE_{\B_1}}\times\O_{\tildeE_{\bar{\B_1}}}\times\O_{\tildeE_{\p_2}}\cong\Z_2^2\times\Z_2(\sqrt{\tDelta})\mid x_1+x_2+x_3+x_3'\in 2\Z_2\},$$
with quadratic form
$$Q(x)=\frac{x_1x_2}{\sqrt{\tildeD}}+\frac{x_3x_3'}{-\sqrt{\tildeD}}=Q_{\p_1}(x_1,x_2)+Q_{\p_2}(x_3).$$

The embedding from $M$ to $M_2$ is given by 
$$x\mapsto (\sigma_1(x),\sigma_1(\bar{x}),\sigma_2(x)).$$
So
$$\varphi_{\mu_1,2}=\mathrm{char}(M_2)=\varphi_{\p_1,0}\varphi_{\p_2,0}+\varphi_{\p_1,1}\varphi_{\p_2,1},$$
where $S_a=\{(x_1,x_2)\in\Z_2^2\mid x_1+x_2\equiv a~(\mod~2)\}$, $\varphi_{\p_1,a}=\mathrm{char}(S_a)\in S(\Q_2^2)\cong S(\tildeE_{\p_1})$ and 
$S_a'=\{x_3\in\Z_2(\sqrt{\tDelta})^2\mid x_3+x_3'\equiv a~(\mod~2)\}$, $\varphi_{\p_2,a}=\mathrm{char}(S_a')\in S(\Q_2(\sqrt{\tDelta}))\cong S(\tildeE_{\p_2})$.

Correspondingly, we have
$$\varphi_{\mu_1}=\varphi_{\mu_1,0}+\varphi_{\mu_1,1},\quad a(t,\varphi_{\mu_1})=a(t,\varphi_{\mu_1,0})+a(t,\varphi_{\mu_1,1}),$$
where $\varphi_{\mu_1,i}=\varphi_{\p_1,i}\varphi_{\p_2,i}\prod_{\p\nmid 2}\varphi_{\p}$.

Proposition \ref{prop:Diff} implies
$$a(t,\varphi_{\mu_1,i})=-4\sum_{\p\text{ inert in }\tildeE/\tildeF}\frac{1+\ord_{\p}(t\sqrt{\tildeD})}{2}\prod_{\q\nmid 2}\rho_{\q}(t\p^{-1}\partial_{\tildeF})\prod_{j=1}^2\frac{W_{t\sqrt{\tildeD},\p_j}^{*,\psi'}(0,\varphi_{\p_j,i})}{\gamma(W_{\p_j})}\log(\mathrm{N}(\p)).$$

\underline{Case (III)}: Write
$$2\text{ inert in }\tildeF,\quad 2\O_{\tildeE}=\B\bar{\B}\text{ splits in }\tildeE.$$

Let $\sqrt{\tDelta}\in\Z_2(\sqrt{\tildeD})$ be some prefixed square roots of $\tDelta$. We identify $\tildeF_2,~\tildeE_{\B},~\tildeE_{\bar{B}}$ with $\Q_2(\sqrt{\tildeD})$ as follows.

\begin{eqnarray*}
\tildeF_2\cong\Q_2(\sqrt{\tildeD}),&&\sqrt{\tildeD}\mapsto\sqrt{\tildeD},\\
\tildeE_{\B}\cong\Q_2(\sqrt{\tildeD}),&&\sqrt{\tildeD}\mapsto\sqrt{\tildeD},~\sqrt{\tDelta}\mapsto\sqrt{\tDelta},\\
\tildeE_{\bar{\B}}\cong\Q_2(\sqrt{\tildeD}),&&\sqrt{\tildeD}\mapsto\sqrt{\tildeD},~\sqrt{\tDelta}\mapsto -\sqrt{\tDelta}
\end{eqnarray*}

With this identification, we can check that $M_2$ is given by
$$M_2=\{(x_1,x_2)\in\O_{\tildeE_{\B}}\times\O_{\tildeE_{\bar{\B}}}\cong\Z_2(\sqrt{\tildeD})^2\mid x_1+x_2+x_1'+x_2'\in 2\Z_2\},$$
with quadratic form
$$Q(x)=\frac{x_1x_2}{\sqrt{\tildeD}}-\frac{x_1'x_2'}{\sqrt{\tildeD}}=Q_{\p}(x_1,x_2)-Q_{\p}(x_1',x_2').$$

The embedding from $M$ to $M_2$ is given by 
$$x\mapsto (\sigma_1(x),\sigma_2(x)).$$
So
$$\varphi_{\mu_1,2}=\mathrm{char}(M_2)=\varphi_{2,0}+\varphi_{2,1},$$
where $S_a''=\{(x_1,x_2)\in\Z_2(\sqrt{\tildeD})^2\mid x_1+x_2\equiv a~(\mod~2)\}$ and $\varphi_{2,a}=\mathrm{char}(S_a'')\in S(\Q_2(\sqrt{\tildeD})^2)\cong S(\tildeE_2).$

Correspondingly, we have
$$\varphi_{\mu_1}=\varphi_{\mu_1,0}+\varphi_{\mu_1,1},\quad a(t,\varphi_{\mu_1})=a(t,\varphi_{\mu_1,0})+a(t,\varphi_{\mu_1,1}),$$
where $\varphi_{\mu_1,i}=\varphi_{2,i}\prod_{\p\nmid 2}\varphi_{\p}$.

Proposition \ref{prop:Diff} implies
$$a(t,\varphi_{\mu_1,i})=-4\sum_{\p\text{ inert in }\tildeE/\tildeF}\frac{1+\ord_{\p}(t\sqrt{\tildeD})}{2}\prod_{\q\nmid 2}\rho_{\q}(t\p^{-1}\partial_{\tildeF})\prod_{j=1}^2\frac{W_{t\sqrt{\tildeD},\p_j}^{*,\psi'}(0,\varphi_{\p_j,i})}{\gamma(W_{\p_j})}\log(\mathrm{N}(\p)).$$

\section{Unitary Case}
\label{cha:unitary}
As a side note, our general theory of CM values also works out for unitary Shimura varieties. Here is the setup.

Let $k=\Q(\sqrt{D})$ be an imaginary quadratic field, where $D<0$ is the discriminant of $k/\Q$ and we fix an embedding of $k$ into $\C$ with $\im(\sqrt{D})>0$. Let $(V,(,)_V)$ be a hermitian space over $k$ of signature $(n,1)$ for some positive integer $n$. Let $\tildeV$ be the underlying $\Q$-vector space of $V$ with associated bilinear form $(x,y)_{\tildeV}=\tr_{k/\Q}(x,y)_V$. There is a natural isomorphism between the following two spaces
\begin{eqnarray*}
V_{\C}=V\otimes_k\C&\stackrel{\sim}{\to}&\tildeV_{\R}=\tildeV\otimes_{\Q}\R\\
1\otimes i&\mapsto&\sqrt{D}\otimes\dfrac{1}{\sqrt{|D|}}
\end{eqnarray*}
with the above complex structure $J$ on $\tildeV_{\R}$.

The associated symmetric spaces are
$$\D=\{\text{negative complex 1-lines in $V_{\C}$}\}$$
and
$$\tildeD=\{\text{oriented negative real 2-planes in $\tildeV_{\R}$}\}.$$

An orientation on a negative 2-plane $z\in\tildeD$ is equivalent to giving a choice of a complex structure $j_z$ on $z$. Then there is a natural embedding
$$\xi_{\D}:\D\hookrightarrow\tildeD,\quad z\mapsto\tilde{z},~J|_{\tilde{z}},$$
where $\tilde{z}$ is the underlying real 2-planes of $z$.

Let $G=\U(V)$ be the indefinite unitary group of signature $(n,1)$
$$G(R)=\U(V)(R)=\{g\in\GL(V\otimes_{\Q}R)\mid (gx,gy)_V=(x,y)_V\}.$$
For $\tildeV$, there is still the indefinite orthogonal group of signature $(2n,2)$
$$\O(V)(R)=\{g\in\GL(\tildeV\otimes_{\Q}R)\mid (gx,gy)_{\tildeV}=(x,y)_{\tildeV}\}.$$
Clearly, by our definition of $(,)_{\tildeV}$, $g\in G$ will automatically preserve the bilinear form $(,)_{\tildeV}$. Therefore, there is a natural embedding between the two groups
$$\xi_G:\U(V)\hookrightarrow\O(\tildeV).$$
Further, since $\U(V)$ is connected, the image of $\xi_G$ should also lie in one of the connected component of $\O(V)$, which should be $\SO(V)$ in this case. So let us denote $\SO(V)$ by $\tildeG$ and $\xi_G:\U(V)\hookrightarrow\SO(\tildeV)$.

Finally, let $K\subset G(\A_f)$ be a compact open subgroup. Let $X_K^u$ be the canonical model of unitary Shimura variety over $\Q$ associated to Shimura datum $(G,\D)$ whose $\C$-points are
$$X_K^u(\C)=G(\Q)\bs (\D\times G(\A_f)/K).$$
Similarly, for Shimura datum ($\tildeG,\tildeD$), we have our orthogonal Shimura variety $X_K^o$. Although there does not exist an embedding between $X_K^u$ and $X_K^o$, there do exist relations between the CM points and Green functions on both sides. With these two main ingredients ready, it will be sufficient to prove our main theorem in the unitary case, derived from the orthogonal case.

Let $d\leq n/2$ and $E$ be a CM number field with totally real field $F$ of degree $d+1$ and a 1-dimensional $E$-hermitian space $(W,(,)_W)$ of signature
$$\sig(W)=((0,1),(1,0),\dots,(1,0))$$
with respect to the $d+1$ $\R$-embeddings $\{\sigma_j\}_{j=0}^d$ such that there exists a positive definite subspace $(V_0,(,)_V|_{V_0})$ of $(V,(,)_V)$ of dimension $n-d$ satisfying
$$V=V_0\oplus\Res_{E/\Q}W.$$

Let $T^u=\Res_{F/\Q}\U(W)$. There is a homomorphism
$$T^u=\Res_{F/\Q}\U(W)\to\U(V)=G$$
as algebraic groups over $\Q$. For the orthogonal side, let $\tildeW$ be the underlying 2-dimensional $F$-space of $E$ with bilinear form $(x,y)_{\tildeW}=\tr_{E/F}(x,y)_W$ of signature
$$\sig(\tildeW)=((0,2),(2,0),\dots,(2,0)).$$
We already have $T^o=\Res_{F/\Q}\SO(W)$ as in Section \ref{cha:SV}. By an easy computation, we have
$$T^u(\Q)=E^1=T^o(\Q).$$
In other words, if we fix $K\subset G(\A_f)$ a compact open subgroup, the CM cycles $Z^u(T^u,h_0,g)_K$ defined as in (\ref{eqn:CM}) for the unitary Shimura variety $X_K^u$ all come from the CM cycles $Z^o(T^o,\xi_{G,\R}\circ h_0,\xi_{G,\A_f}(g))_{\xi_{G,\A_f}(K)}$ defined for the corresponding orthogonal Shimura variety $X_K^o$. Recall the definition of $Z(W)$ in (\ref{eqn:cycle}), we can obtain
$$Z^u(W)=Z^o(\tildeW).$$

For Green function as the regularized theta integral, recall its definition \ref{eqn:Green} in Chapter \ref{cha:Green}, what matters are the harmonic weak Maass form and the theta kernel. Since we have embedded our group $G=\U(V)$ into $\tildeG=\SO(V)$ via $\xi_G$, we can still use the reductive pair $(\O(V),\SL_2)$ for our $(V,(,)_V)$. So the Weil representation $\omega=\omega_{\psi}$ is the same as in the orthogonal case. This, on one hand, means that the definition of harmonic weak Maass forms stay the same. On the other hand, we can also make sure that the action of $\SL_2(\R)$ on $S(V(\R))$ is unchanged in (\ref{eqn:theta}). For the action of $G(\A_f)$ on $S(V(\A_f))$, as noted in the remark under (\ref{eqn:theta}), the action of $g\in G(\A_f)$ is actually via its image in $\SO(V)$. This also stays unchanged since we have the embedding $\xi_G:\U(V)\hookrightarrow\SO(V)$. Therefore, the Green function is exactly the same as in the orthogonal case. 

Now we have
\begin{thm}
For a $K$-invariant harmonic weak Maass form $f\in H_{1-n/2,\rhobar}$ with $f=f^++f^-$ as in (\ref{eqn:Fourier}) and with notation as above,
$$\Phi^u(Z^u(W),f)=\Phi^o(Z^o(\tildeW),f)=\frac{\deg(Z(T^u,z_0^{\pm}))}{\Lambda(0,\chi)}\left(\CT[\langle f^+,\theta_0\otimes\E_W\rangle]-\L^{*,\prime}_W(0,\xi(f))\right).$$
\end{thm}

\newpage
\appendix
\section{Table of Input Modular Forms}
\label{app:Lippolt}
For the ten classical even theta constants, Lippolt \cite{Lip08} have found the following corresponding input modular forms to express them as Borcherds products.

First, following Lippolt \cite{Lip08}, we would like to introduce the following classical theta functions.
\begin{eqnarray*}
\vartheta(\tau)&=&\sum_{n\in\Z}e^{\pi in^2\tau}\\
\tilde{\vartheta}(\tau)&=&\sum_{n\in\Z}(-1)^ne^{\pi in^2\tau}\\
\tilde{\tilde{\vartheta}}(\tau)&=&\sum_{n\in\Z}e^{\pi i\left(n+\frac{1}{2}\right)^2\tau}
\end{eqnarray*}

Now as mentioned in Section \ref{sec:theta}, we have
\begin{eqnarray*}
\u&=&\frac{1}{2\vartheta}+\frac{1}{2\tilde{\vartheta}}=1+4q+14q^2+40q^3+100q^4+232q^5+504q^6+\cdots\\
\v&=&\frac{1}{2\vartheta}-\frac{1}{2\tilde{\vartheta}}=-2q^{\frac{1}{2}}-8q^{\frac{3}{2}}-24q^{\frac{5}{2}}-64q^{\frac{7}{2}}-154q^{\frac{9}{2}}-344q^{\frac{11}{2}}-\cdots\\
\w&=&\frac{2}{\tilde{\tilde{\vartheta}}}=q^{-\frac{1}{8}}-q^{\frac{7}{8}}+q^{\frac{15}{8}}-2q^{\frac{23}{8}}+3q^{\frac{31}{8}}-4q^{\frac{39}{8}}+5q^{\frac{47}{8}}-7q^{\frac{55}{8}}+\cdots
\end{eqnarray*}
where $q=e^{2\pi i\tau}$.

Next, we label the basis of $L'/L\cong\left(\Z\Big/2\Z\right)^4\oplus\left(\Z\Big/4\Z\right)$ as the following
\begin{table}[H]
\centering
\caption{$\mu(i)$\label{tab:label}}
\begin{tabular}{>{$}l<{$}>{$}l<{$}>{$}l<{$}>{$}l<{$}}
\hline
\mbox{\phantom{a}}1:\left(0,0,0,0,0\right)&
\mbox{\phantom{a}}2:\left(0,0,0,\dfrac{1}{2},0\right)&
\mbox{\phantom{a}}3:\left(0,0,\dfrac{1}{2},0,0\right)&
\mbox{\phantom{a}}4:\left(0,0,\dfrac{1}{2},\dfrac{1}{2},\dfrac{2}{4}\right)\\
\mbox{\phantom{a}}5:\left(0,\dfrac{1}{2},0,0,0\right)&
\mbox{\phantom{a}}6:\left(0,\dfrac{1}{2},0,\dfrac{1}{2},0\right)&
\mbox{\phantom{a}}7:\left(0,\dfrac{1}{2},\dfrac{1}{2},0,0\right)&
\mbox{\phantom{a}}8:\left(0,\dfrac{1}{2},\dfrac{1}{2},\dfrac{1}{2},\dfrac{2}{4}\right)\\
\mbox{\phantom{a}}9:\left(\dfrac{1}{2},0,0,0,0\right)&
10:\left(\dfrac{1}{2},0,0,\dfrac{1}{2},0\right)&
11:\left(\dfrac{1}{2},0,\dfrac{1}{2},0,0\right)&
12:\left(\dfrac{1}{2},0,\dfrac{1}{2},\dfrac{1}{2},\dfrac{2}{4}\right)\\
13:\left(\dfrac{1}{2},\dfrac{1}{2},0,0,\dfrac{2}{4}\right)&
14:\left(\dfrac{1}{2},\dfrac{1}{2},0,\dfrac{1}{2},\dfrac{2}{4}\right)&
15:\left(\dfrac{1}{2},\dfrac{1}{2},\dfrac{1}{2},0,\dfrac{2}{4}\right)&
16:\left(\dfrac{1}{2},\dfrac{1}{2},\dfrac{1}{2},\dfrac{1}{2},0\right)\\
17:\left(0,0,\dfrac{1}{2},\dfrac{1}{2},\dfrac{1}{4}\right)&
18:\left(0,0,\dfrac{1}{2},\dfrac{1}{2},\dfrac{3}{4}\right)&
19:\left(0,\dfrac{1}{2},\dfrac{1}{2},\dfrac{1}{2},\dfrac{1}{4}\right)&
20:\left(0,\dfrac{1}{2},\dfrac{1}{2},\dfrac{1}{2},\dfrac{3}{4}\right)\\
21:\left(\dfrac{1}{2},0,\dfrac{1}{2},\dfrac{1}{2},\dfrac{1}{4}\right)&
22:\left(\dfrac{1}{2},0,\dfrac{1}{2},\dfrac{1}{2},\dfrac{3}{4}\right)&
23:\left(\dfrac{1}{2},\dfrac{1}{2},0,0,\dfrac{1}{4}\right)&
24:\left(\dfrac{1}{2},\dfrac{1}{2},0,0,\dfrac{3}{4}\right)\\
25:\left(\dfrac{1}{2},\dfrac{1}{2},0,\dfrac{1}{2},\dfrac{1}{4}\right)&
26:\left(\dfrac{1}{2},\dfrac{1}{2},0,\dfrac{1}{2},\dfrac{3}{4}\right)&
27:\left(\dfrac{1}{2},\dfrac{1}{2},\dfrac{1}{2},0,\dfrac{1}{4}\right)&
28:\left(\dfrac{1}{2},\dfrac{1}{2},\dfrac{1}{2},0,\dfrac{3}{4}\right)\\
29:\left(0,0,0,0,\dfrac{2}{4}\right)&
30:\left(0,0,0,\dfrac{1}{2},\dfrac{2}{4}\right)&
31:\left(0,0,\dfrac{1}{2},0,\dfrac{2}{4}\right)&
32:\left(0,0,\dfrac{1}{2},\dfrac{1}{2},0\right)\\
33:\left(0,\dfrac{1}{2},0,0,\dfrac{2}{4}\right)&
34:\left(0,\dfrac{1}{2},0,\dfrac{1}{2},\dfrac{2}{4}\right)&
35:\left(0,\dfrac{1}{2},\dfrac{1}{2},0,\dfrac{2}{4}\right)&
36:\left(0,\dfrac{1}{2},\dfrac{1}{2},\dfrac{1}{2},0\right)\\
37:\left(\dfrac{1}{2},0,0,0,\dfrac{2}{4}\right)&
38:\left(\dfrac{1}{2},0,0,\dfrac{1}{2},\dfrac{2}{4}\right)&
39:\left(\dfrac{1}{2},0,\dfrac{1}{2},0,\dfrac{2}{4}\right)&
40:\left(\dfrac{1}{2},0,\dfrac{1}{2},\dfrac{1}{2},0\right)\\
41:\left(\dfrac{1}{2},\dfrac{1}{2},0,0,0\right)&
42:\left(\dfrac{1}{2},\dfrac{1}{2},0,\dfrac{1}{2},0\right)&
43:\left(\dfrac{1}{2},\dfrac{1}{2},\dfrac{1}{2},0,0\right)&
44:\left(\dfrac{1}{2},\dfrac{1}{2},\dfrac{1}{2},\dfrac{1}{2},\dfrac{2}{4}\right)\\
45:\left(0,0,0,0,\dfrac{1}{4}\right)&
46:\left(0,0,0,0,\dfrac{3}{4}\right)&
47:\left(0,0,0,\dfrac{1}{2},\dfrac{1}{4}\right)&
48:\left(0,0,0,\dfrac{1}{2},\dfrac{3}{4}\right)\\
49:\left(0,0,\dfrac{1}{2},0,\dfrac{1}{4}\right)&
50:\left(0,0,\dfrac{1}{2},0,\dfrac{3}{4}\right)&
51:\left(0,\dfrac{1}{2},0,0,\dfrac{1}{4}\right)&
52:\left(0,\dfrac{1}{2},0,0,\dfrac{3}{4}\right)\\
53:\left(0,\dfrac{1}{2},0,\dfrac{1}{2},\dfrac{1}{4}\right)&
54:\left(0,\dfrac{1}{2},0,\dfrac{1}{2},\dfrac{3}{4}\right)&
55:\left(0,\dfrac{1}{2},\dfrac{1}{2},0,\dfrac{1}{4}\right)&
56:\left(0,\dfrac{1}{2},\dfrac{1}{2},0,\dfrac{3}{4}\right)\\
57:\left(\dfrac{1}{2},0,0,0,\dfrac{1}{4}\right)&
58:\left(\dfrac{1}{2},0,0,0,\dfrac{3}{4}\right)&
59:\left(\dfrac{1}{2},0,0,\dfrac{1}{2},\dfrac{1}{4}\right)&
60:\left(\dfrac{1}{2},0,0,\dfrac{1}{2},\dfrac{3}{4}\right)\\
61:\left(\dfrac{1}{2},0,\dfrac{1}{2},0,\dfrac{1}{4}\right)&
62:\left(\dfrac{1}{2},0,\dfrac{1}{2},0,\dfrac{3}{4}\right)&
63:\left(\dfrac{1}{2},\dfrac{1}{2},\dfrac{1}{2},\dfrac{1}{2},\dfrac{1}{4}\right)&
64:\left(\dfrac{1}{2},\dfrac{1}{2},\dfrac{1}{2},\dfrac{1}{2},\dfrac{3}{4}\right)
\end{tabular} 
\end{table}

Here are the input data for $f_{\va,\vb}$'s.

1) $f_{1,1,1,1}$ has the following coeffients in $S_L$:
\begin{table}[H]
\caption{$f_{1,1,1,1}$\label{tab:first}}
\begin{tabular}{>{$}l<{$}>{$}l<{$}>{$}l<{$}>{$}l<{$}>{$}l<{$}>{$}l<{$}>{$}l<{$}>{$}l<{$}}
\hline
f_1=\u&f_2=\u&f_3=\u&f_4=-\u&f_5=\u&f_6=\u&f_7=\u&f_8=-\u\\
f_9=\u&f_{10}=\u&f_{11}=\u&f_{12}=-\u&f_{13}=-\u&f_{14}=-\u&f_{15}=-\u&f_{16}=\u\\
f_{17}=0&f_{18}=0&f_{19}=0&f_{20}=0&f_{21}=0&f_{22}=0&f_{23}=0&f_{24}=0\\
f_{25}=0&f_{26}=0&f_{27}=0&f_{28}=0&f_{29}=\v&f_{30}=\v&f_{31}=\v&f_{32}=-\v\\
f_{33}=\v&f_{34}=\v&f_{35}=\v&f_{36}=-\v&f_{37}=\v&f_{38}=\v&f_{39}=\v&f_{40}=-\v\\
f_{41}=-\v&f_{42}=-\v&f_{43}=-\v&f_{44}=\v&f_{45}=\w&f_{46}=\w&f_{47}=0&f_{48}=0\\
f_{49}=0&f_{50}=0&f_{51}=0&f_{52}=0&f_{53}=0&f_{54}=0&f_{55}=0&f_{56}=0\\
f_{57}=0&f_{58}=0&f_{59}=0&f_{60}=0&f_{61}=0&f_{62}=0&f_{63}=0&f_{64}=0
\end{tabular}
\end{table}

2) $f_{0,1,1,0}$ has the following coeffients in $S_L$:
\begin{table}[H]
\caption{$f_{0,1,1,0}$}
\begin{tabular}{>{$}l<{$}>{$}l<{$}>{$}l<{$}>{$}l<{$}>{$}l<{$}>{$}l<{$}>{$}l<{$}>{$}l<{$}}
\hline
f_1=\u&f_2=\u&f_3=-\u&f_4=\u&f_5=\u&f_6=\u&f_7=-\u&f_8=\u\\
f_9=\u&f_{10}=\u&f_{11}=-\u&f_{12}=\u&f_{13}=-\u&f_{14}=-\u&f_{15}=\u&f_{16}=-\u\\
f_{17}=0&f_{18}=0&f_{19}=0&f_{20}=0&f_{21}=0&f_{22}=0&f_{23}=0&f_{24}=0\\
f_{25}=0&f_{26}=0&f_{27}=0&f_{28}=0&f_{29}=\v&f_{30}=\v&f_{31}=-\v&f_{32}=\v\\
f_{33}=\v&f_{34}=\v&f_{35}=-\v&f_{36}=\v&f_{37}=\v&f_{38}=\v&f_{39}=-\v&f_{40}=\v\\
f_{41}=-\v&f_{42}=-\v&f_{43}=\v&f_{44}=-\v&f_{45}=0&f_{46}=0&f_{47}=\w&f_{48}=\w\\
f_{49}=0&f_{50}=0&f_{51}=0&f_{52}=0&f_{53}=0&f_{54}=0&f_{55}=0&f_{56}=0\\
f_{57}=0&f_{58}=0&f_{59}=0&f_{60}=0&f_{61}=0&f_{62}=0&f_{63}=0&f_{64}=0
\end{tabular}
\end{table}

3) $f_{1,0,0,1}$ has the following coeffients in $S_L$:
\begin{table}[H]
\caption{$f_{1,0,0,1}$}
\begin{tabular}{>{$}l<{$}>{$}l<{$}>{$}l<{$}>{$}l<{$}>{$}l<{$}>{$}l<{$}>{$}l<{$}>{$}l<{$}}
\hline
f_1=\u&f_2=-\u&f_3=\u&f_4=\u&f_5=\u&f_6=-\u&f_7=\u&f_8=\u\\
f_9=\u&f_{10}=-\u&f_{11}=\u&f_{12}=\u&f_{13}=-\u&f_{14}=\u&f_{15}=-\u&f_{16}=-\u\\
f_{17}=0&f_{18}=0&f_{19}=0&f_{20}=0&f_{21}=0&f_{22}=0&f_{23}=0&f_{24}=0\\
f_{25}=0&f_{26}=0&f_{27}=0&f_{28}=0&f_{29}=\v&f_{30}=-\v&f_{31}=\v&f_{32}=\v\\
f_{33}=\v&f_{34}=-\v&f_{35}=\v&f_{36}=\v&f_{37}=\v&f_{38}=-\v&f_{39}=\v&f_{40}=\v\\
f_{41}=-\v&f_{42}=\v&f_{43}=-\v&f_{44}=-\v&f_{45}=0&f_{46}=0&f_{47}=0&f_{48}=0\\
f_{49}=\w&f_{50}=\w&f_{51}=0&f_{52}=0&f_{53}=0&f_{54}=0&f_{55}=0&f_{56}=0\\
f_{57}=0&f_{58}=0&f_{59}=0&f_{60}=0&f_{61}=0&f_{62}=0&f_{63}=0&f_{64}=0
\end{tabular}
\end{table}

4) $f_{0,0,1,1}$ has the following coeffients in $S_L$:
\begin{table}[H]
\caption{$f_{0,0,1,1}$}
\begin{tabular}{>{$}l<{$}>{$}l<{$}>{$}l<{$}>{$}l<{$}>{$}l<{$}>{$}l<{$}>{$}l<{$}>{$}l<{$}}
\hline
f_1=\u&f_2=\u&f_3=\u&f_4=-\u&f_5=\u&f_6=\u&f_7=\u&f_8=-\u\\
f_9=-\u&f_{10}=-\u&f_{11}=-\u&f_{12}=\u&f_{13}=\u&f_{14}=\u&f_{15}=\u&f_{16}=-\u\\
f_{17}=0&f_{18}=0&f_{19}=0&f_{20}=0&f_{21}=0&f_{22}=0&f_{23}=0&f_{24}=0\\
f_{25}=0&f_{26}=0&f_{27}=0&f_{28}=0&f_{29}=\v&f_{30}=\v&f_{31}=\v&f_{32}=-\v\\
f_{33}=\v&f_{34}=\v&f_{35}=\v&f_{36}=-\v&f_{37}=-\v&f_{38}=-\v&f_{39}=-\v&f_{40}=\v\\
f_{41}=\v&f_{42}=\v&f_{43}=\v&f_{44}=-\v&f_{45}=0&f_{46}=0&f_{47}=0&f_{48}=0\\
f_{49}=0&f_{50}=0&f_{51}=\w&f_{52}=\w&f_{53}=0&f_{54}=0&f_{55}=0&f_{56}=0\\
f_{57}=0&f_{58}=0&f_{59}=0&f_{60}=0&f_{61}=0&f_{62}=0&f_{63}=0&f_{64}=0
\end{tabular}
\end{table}

5) $f_{0,0,1,0}$ has the following coeffients in $S_L$:
\begin{table}[H]
\caption{$f_{0,0,1,0}$}
\begin{tabular}{>{$}l<{$}>{$}l<{$}>{$}l<{$}>{$}l<{$}>{$}l<{$}>{$}l<{$}>{$}l<{$}>{$}l<{$}}
\hline
f_1=\u&f_2=\u&f_3=-\u&f_4=\u&f_5=\u&f_6=\u&f_7=-\u&f_8=\u\\
f_9=-\u&f_{10}=-\u&f_{11}=\u&f_{12}=\u&f_{13}=\u&f_{14}=\u&f_{15}=-\u&f_{16}=\u\\
f_{17}=0&f_{18}=0&f_{19}=0&f_{20}=0&f_{21}=0&f_{22}=0&f_{23}=0&f_{24}=0\\
f_{25}=0&f_{26}=0&f_{27}=0&f_{28}=0&f_{29}=\v&f_{30}=\v&f_{31}=-\v&f_{32}=\v\\
f_{33}=\v&f_{34}=\v&f_{35}=-\v&f_{36}=\v&f_{37}=-\v&f_{38}=-\v&f_{39}=\v&f_{40}=-\v\\
f_{41}=\v&f_{42}=\v&f_{43}=-\v&f_{44}=\v&f_{45}=0&f_{46}=0&f_{47}=0&f_{48}=0\\
f_{49}=0&f_{50}=0&f_{51}=0&f_{52}=0&f_{53}=\w&f_{54}=\w&f_{55}=0&f_{56}=0\\
f_{57}=0&f_{58}=0&f_{59}=0&f_{60}=0&f_{61}=0&f_{62}=0&f_{63}=0&f_{64}=0
\end{tabular}
\end{table}

6) $f_{0,0,0,1}$ has the following coeffients in $S_L$:
\begin{table}[H]
\caption{$f_{0,0,0,1}$}
\begin{tabular}{>{$}l<{$}>{$}l<{$}>{$}l<{$}>{$}l<{$}>{$}l<{$}>{$}l<{$}>{$}l<{$}>{$}l<{$}}
\hline
f_1=\u&f_2=-\u&f_3=\u&f_4=\u&f_5=\u&f_6=-\u&f_7=\u&f_8=\u\\
f_9=-\u&f_{10}=\u&f_{11}=-\u&f_{12}=-\u&f_{13}=\u&f_{14}=-\u&f_{15}=\u&f_{16}=\u\\
f_{17}=0&f_{18}=0&f_{19}=0&f_{20}=0&f_{21}=0&f_{22}=0&f_{23}=0&f_{24}=0\\
f_{25}=0&f_{26}=0&f_{27}=0&f_{28}=0&f_{29}=\v&f_{30}=-\v&f_{31}=\v&f_{32}=\v\\
f_{33}=\v&f_{34}=-\v&f_{35}=\v&f_{36}=\v&f_{37}=-\v&f_{38}=\v&f_{39}=-\v&f_{40}=-\v\\
f_{41}=\v&f_{42}=-\v&f_{43}=\v&f_{44}=\v&f_{45}=0&f_{46}=0&f_{47}=0&f_{48}=0\\
f_{49}=0&f_{50}=0&f_{51}=0&f_{52}=0&f_{53}=0&f_{54}=0&f_{55}=\w&f_{56}=\w\\
f_{57}=0&f_{58}=0&f_{59}=0&f_{60}=0&f_{61}=0&f_{62}=0&f_{63}=0&f_{64}=0
\end{tabular}
\end{table}

7) $f_{1,1,0,0}$ has the following coeffients in $S_L$:
\begin{table}[H]
\caption{$f_{1,1,0,0}$}
\begin{tabular}{>{$}l<{$}>{$}l<{$}>{$}l<{$}>{$}l<{$}>{$}l<{$}>{$}l<{$}>{$}l<{$}>{$}l<{$}}
\hline
f_1=\u&f_2=\u&f_3=\u&f_4=-\u&f_5=-\u&f_6=-\u&f_7=-\u&f_8=\u\\
f_9=\u&f_{10}=\u&f_{11}=\u&f_{12}=-\u&f_{13}=\u&f_{14}=\u&f_{15}=\u&f_{16}=-\u\\
f_{17}=0&f_{18}=0&f_{19}=0&f_{20}=0&f_{21}=0&f_{22}=0&f_{23}=0&f_{24}=0\\
f_{25}=0&f_{26}=0&f_{27}=0&f_{28}=0&f_{29}=\v&f_{30}=\v&f_{31}=\v&f_{32}=-\v\\
f_{33}=-\v&f_{34}=-\v&f_{35}=-\v&f_{36}=\v&f_{37}=\v&f_{38}=\v&f_{39}=\v&f_{40}=-\v\\
f_{41}=\v&f_{42}=\v&f_{43}=\v&f_{44}=-\v&f_{45}=0&f_{46}=0&f_{47}=0&f_{48}=0\\
f_{49}=0&f_{50}=0&f_{51}=0&f_{52}=0&f_{53}=0&f_{54}=0&f_{55}=0&f_{56}=0\\
f_{57}=\w&f_{58}=\w&f_{59}=0&f_{60}=0&f_{61}=0&f_{62}=0&f_{63}=0&f_{64}=0
\end{tabular}
\end{table}

8) $f_{0,1,0,0}$ has the following coeffients in $S_L$:
\begin{table}[H]
\caption{$f_{0,1,0,0}$}
\begin{tabular}{>{$}l<{$}>{$}l<{$}>{$}l<{$}>{$}l<{$}>{$}l<{$}>{$}l<{$}>{$}l<{$}>{$}l<{$}}
\hline
f_1=\u&f_2=\u&f_3=-\u&f_4=\u&f_5=-\u&f_6=-\u&f_7=\u&f_8=-\u\\
f_9=\u&f_{10}=\u&f_{11}=-\u&f_{12}=\u&f_{13}=\u&f_{14}=\u&f_{15}=-\u&f_{16}=\u\\
f_{17}=0&f_{18}=0&f_{19}=0&f_{20}=0&f_{21}=0&f_{22}=0&f_{23}=0&f_{24}=0\\
f_{25}=0&f_{26}=0&f_{27}=0&f_{28}=0&f_{29}=\v&f_{30}=\v&f_{31}=-\v&f_{32}=\v\\
f_{33}=-\v&f_{34}=-\v&f_{35}=\v&f_{36}=-\v&f_{37}=\v&f_{38}=\v&f_{39}=-\v&f_{40}=\v\\
f_{41}=\v&f_{42}=\v&f_{43}=-\v&f_{44}=\v&f_{45}=0&f_{46}=0&f_{47}=0&f_{48}=0\\
f_{49}=0&f_{50}=0&f_{51}=0&f_{52}=0&f_{53}=0&f_{54}=0&f_{55}=0&f_{56}=0\\
f_{57}=0&f_{58}=0&f_{59}=\w&f_{60}=\w&f_{61}=0&f_{62}=0&f_{63}=0&f_{64}=0
\end{tabular}
\end{table}

9) $f_{1,0,0,0}$ has the following coeffients in $S_L$:
\begin{table}[H]
\caption{$f_{1,0,0,0}$}
\begin{tabular}{>{$}l<{$}>{$}l<{$}>{$}l<{$}>{$}l<{$}>{$}l<{$}>{$}l<{$}>{$}l<{$}>{$}l<{$}}
\hline
f_1=\u&f_2=-\u&f_3=\u&f_4=\u&f_5=-\u&f_6=\u&f_7=-\u&f_8=-\u\\
f_9=\u&f_{10}=-\u&f_{11}=\u&f_{12}=\u&f_{13}=\u&f_{14}=-\u&f_{15}=\u&f_{16}=\u\\
f_{17}=0&f_{18}=0&f_{19}=0&f_{20}=0&f_{21}=0&f_{22}=0&f_{23}=0&f_{24}=0\\
f_{25}=0&f_{26}=0&f_{27}=0&f_{28}=0&f_{29}=\v&f_{30}=-\v&f_{31}=\v&f_{32}=\v\\
f_{33}=-\v&f_{34}=\v&f_{35}=-\v&f_{36}=-\v&f_{37}=\v&f_{38}=-\v&f_{39}=\v&f_{40}=\v\\
f_{41}=\v&f_{42}=-\v&f_{43}=\v&f_{44}=\v&f_{45}=0&f_{46}=0&f_{47}=0&f_{48}=0\\
f_{49}=0&f_{50}=0&f_{51}=0&f_{52}=0&f_{53}=0&f_{54}=0&f_{55}=0&f_{56}=0\\
f_{57}=0&f_{58}=0&f_{59}=0&f_{60}=0&f_{61}=\w&f_{62}=\w&f_{63}=0&f_{64}=0
\end{tabular}
\end{table}

10) $f_{0,0,0,0}$ has the following coeffients in $S_L$:
\begin{table}[H]
\caption{$f_{0,0,0,0}$\label{tab:last}}
\begin{tabular}{>{$}l<{$}>{$}l<{$}>{$}l<{$}>{$}l<{$}>{$}l<{$}>{$}l<{$}>{$}l<{$}>{$}l<{$}}
\hline
f_1=\u&f_2=-\u&f_3=-\u&f_4=-\u&f_5=-\u&f_6=\u&f_7=\u&f_8=\u\\
f_9=-\u&f_{10}=\u&f_{11}=\u&f_{12}=\u&f_{13}=-\u&f_{14}=\u&f_{15}=\u&f_{16}=\u\\
f_{17}=0&f_{18}=0&f_{19}=0&f_{20}=0&f_{21}=0&f_{22}=0&f_{23}=0&f_{24}=0\\
f_{25}=0&f_{26}=0&f_{27}=0&f_{28}=0&f_{29}=\v&f_{30}=-\v&f_{31}=-\v&f_{32}=-\v\\
f_{33}=-\v&f_{34}=\v&f_{35}=\v&f_{36}=\v&f_{37}=-\v&f_{38}=\v&f_{39}=\v&f_{40}=\v\\
f_{41}=-\v&f_{42}=\v&f_{43}=\v&f_{44}=\v&f_{45}=0&f_{46}=0&f_{47}=0&f_{48}=0\\
f_{49}=0&f_{50}=0&f_{51}=0&f_{52}=0&f_{53}=0&f_{54}=0&f_{55}=0&f_{56}=0\\
f_{57}=0&f_{58}=0&f_{59}=0&f_{60}=0&f_{61}=0&f_{62}=0&f_{63}=\w&f_{64}=\w
\end{tabular}
\end{table}

\bibliography{references.bib}

\end{document}